\documentclass[12pt]{amsart}
\usepackage{url}
\usepackage{color}
\usepackage[a4paper,asymmetric]{geometry}
\usepackage{latexsym}
\usepackage{amsthm}
\usepackage{amssymb}
\usepackage{amsfonts}
\usepackage{amsmath}
\newtheorem{theorem}{Theorem}[section]
\newtheorem{thm}[theorem]{Theorem}

\newtheorem{lem}[theorem]{Lemma}
\newtheorem{proposition}[theorem]{Proposition}

\newtheorem{corollary}[theorem]{Corollary}
\newtheorem{assumption}[theorem]{Assumption}
\theoremstyle{definition}

\newtheorem{defn}[theorem]{Definition}
\newtheorem{ex}[theorem]{Example}
\theoremstyle{remark}

\newtheorem{rem}[theorem]{Remark}
\numberwithin{equation}{section}

 \DeclareMathAlphabet{\mathpzc}{OT1}{pzc}{m}{it}
  \newcommand{\mmmm}{\mathfrak{m}}

\def\lan{{\langle}}
\def\ran{{\rangle}}

\def\cal{\mathcal}
 \newcommand{\E}{\mathbb{E}}            % expectation
 \newcommand{\e}{\varepsilon}
 \newcommand{\p}{\partial}

 \newcommand{\Ll}{\langle}
 \newcommand{\Rr}{\rangle}

 \newcommand{\N}{\mathbb{N}}
 \newcommand{\R}{\mathbb{R}}
 \newcommand{\Z}{\mathbb{Z}}
 \newcommand{\PP}{\mathbb{P}}
 \newcommand{\mcl}{\mathcal}
 
 \newcommand{\Be}{\begin{equation}}
 \newcommand{\Ee}{\end{equation}}
 \newcommand{\Bs}{\begin{split}}
 \newcommand{\Es}{\end{split}}
  \newcommand{\Bes}{\begin{equation*}}
 \newcommand{\Ees}{\end{equation*}}
 \newcommand{\BT}{\begin{thm}}
 \newcommand{\ET}{\end{thm}}
 \newcommand{\Bp}{\begin{proof}}
 \newcommand{\Ep}{\end{proof}}
 \newcommand{\BL}{\begin{lem}}
 \newcommand{\EL}{\end{lem}}
 \newcommand{\BP}{\begin{proposition}}
 \newcommand{\EP}{\end{proposition}}
 \newcommand{\BC}{\begin{corollary}}
 \newcommand{\EC}{\end{corollary}}
 \newcommand{\BR}{\begin{rem}}
 \newcommand{\ER}{\end{rem}}
 \newcommand{\BD}{\begin{defn}}
 \newcommand{\ED}{\end{defn}}
 \newcommand{\BI}{\begin{itemize}}
 \newcommand{\EI}{\end{itemize}}
 \newcommand{\eqn}{equation}

 \newcommand{\bg}{\big}

  \newcommand{\const}{\mathop{{\rm const}}\nolimits}

\begin{document}
\title
[Exponential ergodicity and Regularity]
%for equations with
%L\'evy noise]
{Exponential ergodicity and regularity
%in variation
%distance
for equations with L\'evy noise}
\author[E. Priola]{Enrico Priola}
\address{Dipartimento di Matematica, Universit\`a di Torino,
via Carlo Alberto 10 \\ 10123 Torino, Italy} \email{enrico.priola@unito.it}
\author[A. Shirikyan]{Armen Shirikyan}
\address{Department of Mathematics, University of Cergy--Pontoise, CNRS UMR 8088, 2 avenue Adolphe Chauvin, 95302 Cergy--Pontoise, France}
\email{Armen.Shirikyan@u-cergy.fr}
\author[L. Xu]{Lihu Xu}
\address{TU Berlin, Fakult\"{a}t II, Institut f\"{u}r Mathematik,
Str$\alpha \beta$e des 17. Juni 136, D-10623 Berlin, Germany}
\email{xu@math.tu-berlin.de}
\author[J. Zabczyk]{Jerzy Zabczyk}
\address{Institute of Mathematics, Polish Academy of Sciences, P-00-950 Warszawa, Poland}
\email{zabczyk@impan.pl}

\thanks{The first author
was supported by the M.I.U.R. research project Prin 2008
 ``Deterministic and stochastic methods in the study of
 evolution problems''. The third author gratefully acknowledges the
support by Junior program \emph{Stochastics} of Hausdorff Research
Institute for Mathematics. His research is partially supported by
the European Research Council under the European Union's Seventh
Framework Programme (FP7/2007-2013) / ERC grant agreement
nr.~258237. The fourth author gratefully acknowledges the support by the
Polish Ministry of Science and Higher Education grant ``Stochastic
equations in infinite dimensional spaces'' N N201 419039}
%\subjclass[2000]{}
%\keywords{}
%\date{}
%%% ----------------------------------------------------------------------

 \maketitle

\begin{abstract}
\label{abstract} \noindent  We prove exponential convergence to the
invariant measure, in the total variation norm, for  solutions of SDEs
driven by $\alpha$-stable noises in finite and in infinite
dimensions. Two approaches are used.  The first one is based on Harris theorem,
and the second on
Doeblin's coupling argument~\cite{doeblin-1940}. Irreducibility, Lyapunov function techniques, and uniform strong Feller property play an essential role in both
approaches.
 We concentrate on two classes of Markov processes:
  solutions of finite-dimensional equations,
  introduced in~\cite{Pr10},
 with  H\"older continuous drift and
  a general, non-degenerate, symmetric $\alpha$-stable noise,
   and  infinite-dimensional parabolic systems,
 introduced in~\cite{PZ09}, with Lipschitz  drift
 and   cylindrical $\alpha$-stable noise. We show  that if the
nonlinearity is bounded, then the processes are  exponential mixing.
%under the total variation norm.
This improves, in particular,  an earlier result established
in~\cite{PXZ10}  using the weak convergence induced  by the
 Kantorovich--Wasserstein metric.
%with a different method.


\end{abstract}

\tableofcontents

\section{Introduction}
This paper is concerned with ergodic properties of the stochastic equation
\begin{\eqn} \label{e:MaiSPDE}
dX_t=[AX_t+F(X_t)]dt+dZ_t, \quad X_0=x,
\end{\eqn}
both in finite- and infinite-dimensional real Hilbert spaces~$H$.
Here~$A$ is a linear operator, $F$~is a bounded mapping, and~$Z$ is  a
symmetric $\alpha$-stable process.
  Under suitable conditions, we  establish exponential
convergence of the solutions to the invariant measure in the
variation norm. Note that many nonlinear
stochastic PDEs, including semilinear heat equations perturbed by
L\'evy noise, can be written in the form~\eqref{e:MaiSPDE} with an infinite-dimensional phase space~$H$.

Irreducibility and uniform strong Feller properties
 play an essential role in our approach.
  They are established in the paper when the space~$H$ is
finite-dimensional, $Z$ is a non-degenerate symmetric
$\alpha$-stable process,  and~$F$ is an $\eta$-H\"older continuous function
 with $1-\frac \alpha2< \eta \leq 1$ and $1<\alpha<2$.
 Under stronger assumptions on the drift~$F$ and on the noise process $Z$,
   those properties were derived  in~\cite{PZ09} in infinite dimensions. The
  finite-dimensional result established in this paper is an important contribution of independent interest.

Stochastic PDEs  driven by L\'evy noises have been intensively
studied since some time; e.g., see the
papers~\cite{AWZ98,AMR09,PeZa06,Ok08,MaRo09,PZ09,XZ09}, the
book~\cite{PeZa07}, and the references therein. Invariant measures
and long-time asymptotics for stochastic systems driven by L\'evy
noises were studied in a number of papers. In particular, the linear
case ($F \equiv 0$) was investigated in~\cite{sato1,zabczyk83} for finite-dimensional spaces and in~\cite{Ch1987,PrZa09-2,FuXi09} for the infinite dimension. The case of nonlinear equations was
studied in~\cite{Rusinek,PeZa07, Mas, XZ09,XZ10}. However, there are
no many results on ergodicity and exponential mixing (cf.~\cite{XZ10, Ku09, PXZ10}). The paper~\cite{Ku09}  studied the
exponential mixing of finite-dimensional stochastic systems with
jump noises, which include one-dimensional SDEs driven by
$\alpha$-stable noise.

 Some ergodic properties for  SPDEs like~\eqref{e:MaiSPDE}
 were also studied  in~\cite{PXZ10}.
It was proved that if the supremum norm of~$F$ is \emph{small}, then
there exists a unique invariant measure, which is exponential mixing
under the  weak convergence induced by the Kantorovich--Wasserstein
metric.
%weak topology in the space of measures.
Here we improve substantially this
 result, showing that the convergence to the invariant measure holds
 exponentially fast in the total variation norm
 without any smallness assumption on~$F$.
 To prove this result, we have to impose a slightly stronger
 regularity condition on the noise compared to that
of~\cite{PXZ10}; this is, however, a really mild assumption (see Remark~\ref{compare} and Example~\ref{comp}).

 As mentioned before,
 we also establish exponential mixing
  in the total variation norm  for finite-dimensional stochastic
 equations of the form~\eqref{e:MaiSPDE} with a less regular drift term~$F$ and a more general noise~$Z$.
 It seems that, even in  one dimension
  (when $Z$ reduces to a  standard symmetric
rotationally invariant $\alpha$-stable noise), our result on
exponential mixing is new (cf.~\cite{XZ09,Ku09}).

We have two proofs for the exponential mixing results. Even though they give the same result, we included both proofs in the paper since they are based on some additional properties of independent interest, such as exponential estimates for hitting times of balls.
The first proof is based on Harris' theorem, while the other uses
 the classical coupling argument, see Section~\ref{s2.3} and also~\cite{lindvall2002}. In both approaches, irreducibility and uniform
strong Feller property play a crucial role. The Harris  approach
only needs to check some conditions involving Lyapunov functions,
but it is not intuitive. The coupling proof is more involved, but
gives an intuition for understanding the way in which  the
dynamics converges to the ergodic measure.

Let us sketch our methods for proving the  well-posedness and the
structural properties of finite-dimensional stochastic systems,
since it has independent interest. To prove the existence and
pathwise uniqueness of  solutions, we only need to modify slightly the argument in~\cite{Pr10}. We stress that the
condition  $1-\frac \alpha2< \eta \leq 1$ is needed to have
existence and  uniqueness of solutions (cf.~\cite{Pr10}). The
irreducibility and uniform strong Feller property will be
established in the following two steps. First, we prove
irreducibility and (uniform) gradient estimates for
finite-dimensional Ornstein--Uhlenbeck processes driven by non-degenerate
symmetric $\alpha$-stable processes (related  gradient estimates
under different assumptions from ours are given in  the recent paper~\cite{W10}).
Then we proceed as in~\cite{PZ09} and deduce
irreducibility and uniform gradient estimates for  solutions to~\eqref{e:MaiSPDE}.
Note that if $\eta < 1$ then  the deterministic
equation may have many solutions as classical examples show.
Currently, there is a great interest in understanding pathwise
uniqueness for SDEs when~$F$ is not Lipschitz,
%see~\cite{FGP10,DPF10} for Gaussian noises and
see the  references given in~\cite{DPF10, Pr10}.

\smallskip
The paper is organized as follows. In  Section~\ref{s:NM},
we formulate basic structural properties of the solutions
of~\eqref{e:MaiSPDE} and our main ergodic
results---Theorems~\ref{t:MaiThm} and~\ref{t:ErgFin}. In Section~\ref{s:F}, we
concentrate on proving the new structural properties of
finite-dimensional systems.
%, which are of independent interest.
Section~\ref{s3} contains  decay $L_p$-estimates for solutions
of~\eqref{e:MaiSPDE}, which are needed to prove exponential
 ergodicity; here we concentrate on the infinite-dimensional case
 since in finite dimensions these estimates are straightforward.
  %the preparations for proving Theorem \ref{t:MaiThm}.
The two proofs for the exponential mixing of infinite dynamics are
established in Sections~\ref{s:Har} and~\ref{s4}, respectively, the
former applying Harris' theorem and the latter using coupling
argument. Section~\ref{s4} is quite involved, in particular,
exponential estimates for the first hitting time of balls are of
independent interest. In Section~\ref{s6}, we show the exponential
 ergodicity for finite-dimensional  systems
 (Theorem~\ref{t:ErgFin}) in a sketchy way. We have only shown the full
details for the proof of Theorem~\ref{t:MaiThm} concerning SPDEs,
since the finite-dimensional result can be proved by similar and
easier methods.

\medskip
{\bf Acknowledgements.}
 We would like to thank C.~Odasso for patiently discussing with us
his paper~\cite{Od07} and writing a note for us on the proof of
inequality~\eqref{e:TauPro}.
 We also would like to thank M.~Hairer for pointing out to us
the  proof of Theorem~\ref{t:MaiThm} by the Harris  approach.

\section{Main results} \label{s:NM}
\subsection{Notations and assumptions}
Let~$H$  be a real separable Hilbert space with an inner product $\Ll\cdot,\cdot\Rr$
and the corresponding norm~$|\cdot|$. We denote by~$\{e_k\}_{k \geq1}$ an orthonormal basis, so that any vector $x\in H$ can be written as $x=\sum_{k\geq1} x_k e_k$, where $\sum_k|x_k|^2<\infty$.
Denote by~$B_b(H)$ the Banach space  of bounded Borel-measurable functions $f:H \rightarrow \R$ with the supremum norm
$$
\|f\|_{0}:=\sup_{x \in H} |f(x)|.
$$

\smallskip
Let $\mcl B(H)$ be the Borel $\sigma$-algebra on $H$ and let~$\mcl P(H)$ be the set of probabilities on $(H,\mcl B(H))$. Recall that the total variation distance between two measures $\mu_1, \mu_2 \in \mcl P(H)$ is defined by
$$
\|\mu_1-\mu_2\|_{\rm TV}
= \frac 12 \sup_{\stackrel{f \in B_b(H)}{\|f\|_0=1}} |\mu_1(f)-\mu_2(f)|
=\sup_{\Gamma\in\mcl B(H)}|\mu_1(\Gamma)-\mu_2(\Gamma)|.
$$

Let $z(t)$ be a one-dimensional symmetric $\alpha$-stable process with $0<\alpha<2$. Its infinitesimal generator~$\mcl A$ is given by
\begin{equation} \label{e:fraclap}
\mcl A f(x):=\frac{1}{C_{\alpha}} \int_{\mathbb{R}}
\frac{f(y+x)-f(x)}{|y|^{\alpha+1}}dy, \;\; x \in \R,
\end{equation}
where $C_{\alpha}= -\int_{\mathbb{R}} (\cos y-1)\frac{dy}{|y|^{1+\alpha}}$;
see~\cite{sato} and~\cite{ARW00}.
It is well known that~$z(t)$ has the following characteristic function:
\begin{\eqn*}
\E [e^{i \lambda z(t)}]=e^{-t|\lambda|^{\alpha}},
\end{\eqn*}
$t \ge 0$, $\lambda \in \R$. A multidimensional generalization of $z(t)$ is obtained by considering
an $n$-dimensional non-degenerate  symmetric $\alpha $-stable  process
     $Z=(Z_t)$. This is a L\'evy process
     with the additional
     property that
 \begin{align} \label{itol}
 \E[ e^{i \langle  Z_t ,  u\rangle }]  = e^{- t \psi(u)}, \quad
  \psi(u)= - \int_{\R^d} \Big(  e^{i \langle u,y \rangle }
   - 1 - \, { i \langle u,y
\rangle} \, 1_{ \{ |y| \le 1 \} } \, (y) \Big ) \nu (dy),
  \end{align}
 $u \in \R^n$, $t \ge 0$,  where
   the L\'evy (intensity) measure $\nu$
   is  of the form
 \begin{align} \label{spec}
 \nu (D) = \int_{S} \mu(d \xi) \int_0^{\infty} 1_D (r \xi)
 \frac{dr}{r^{1+ \alpha}}, \;\;\; D \in {\cal B}(\R^n),
 \end{align}
 for some symmetric, non-zero  finite measure $\mu$
  concentrated on the
  unit sphere $S =  \{ y \in \R^d \, : \, |y| = 1 \} $
  (see \cite[Theorem 14.3]{sato}).
  %  it follows that
  % $\nu$ is symmetric as well.
Note that  formula~\eqref{spec} implies that
$
  \psi(u)=   c_{\alpha}
    \int_{S} |\langle u,  \xi \rangle |^{\alpha} \mu (d \xi),
    \; u \in \R^n
$
(see also \cite[Theorem 14.13]{sato}).
   The non-degeneracy hypothesis on $Z$ is
   the assumption that
   there exists a positive
 constant~$C_{\alpha}$ such that, for any $u \in \R^n$,
 \begin{align} \label{nondeg}
 \psi (u) \ge C_{\alpha} |u|^{\alpha}.
\end{align}
This is equivalent to the fact that the support of $\mu$ is not contained in a proper linear subspace of $\R^n$ (see~\cite{Pr10} for more details).
   Recall that  the infinitesimal
 generator~$\cal A$ of the  process $Z$ is given
  on the space of
  infinitely differentiable functions with compact support
 $C^{\infty}_c(\R^n)$ by the formula
 $$
 {\cal A}f(x)=\int_{\R^d} \big(  f (x+y) -  f (x)
   - 1_{ \{  |y| \le 1 \} } \, \langle y , D f (x) \rangle \big)
   \, \nu (dy), \;\; f \in C^{\infty}_c(\R^n),
 $$
see \cite[Section 31]{sato}.  Note that
 $Z_t$ $= \sum_{1 \leq j \leq n} \beta_{j} z_j(t) e_j$
(where $\{z_j(t)\}_{1 \leq j \leq n}$ are i.i.d. one-dimensional
symmetric $\alpha$-stable processes) is in particular a
non-degenerate symmetric $\alpha$-stable process if  each $\beta_j
\not =0.$ \vskip 1mm

We will make two sets of assumptions on~\eqref{e:MaiSPDE} depending
on the dimension of the Hilbert space $H$. They are  similar but
more restrictive if $\dim H  = \infty$.

\begin{assumption} \label{a:A2}
 $[\dim H=n < \infty ]$
\begin{itemize}
\item[(A1)]
 $A$ is an $n \times n$ matrix and $\max_{1\leq i \leq n}
Re(\gamma_k)<0$, where $\gamma_1, \ldots, \gamma_n$ are the
eigenvalues of $A$  counted according to their multiplicity.

\item[(A2)]
 $Z = (Z_t)$ is a symmetric  non-degenerate $n$-dimensional
$\alpha$-stable process with $1 <  \alpha<2$.

\item[(A3)]
$F: H \rightarrow H$ is bounded and $\eta$-H\"older continuous with $1-\frac \alpha2< \eta \leq 1$.
\end{itemize}
\end{assumption}

\begin{assumption} \label{a:A}
$[\dim H= \infty]$
\begin{itemize}
\item[(A1)] $A$ is a dissipative operator defined by
$$
A=\sum_{k \geq 1} (-\gamma_k) e_k \otimes e_k
$$
with $0<\gamma_1 \leq \gamma_2 \leq \ldots \leq \gamma_k \leq
\ldots$ and $\gamma_k \rightarrow \infty$ as $k \rightarrow \infty$.
\item[(A2)] $Z_t$ is a cylindrical $\alpha$-stable process with
$Z_t=\sum_{k \geq 1} \beta_k z_k(t) e_k$, where $\{z_k(t)\}_{k \geq
1}$ are i.i.d. symmetric $\alpha$-stable processes with
$0<\alpha<2$ and~$\beta_k$ are positive constants such that  $\sum_{k \geq 1}
\frac{\beta^{\alpha}_k}{\gamma^{1-\alpha \e}_k}<\infty$ for some  $\e \in (0,1)$.
\item[(A3)]
$F: H \rightarrow H$ is Lipschitz and bounded.
\item[(A4)] There exist some $\theta \in (0,1)$ and $C>0$
 so that $\beta_k \geq
C \gamma_k^{-\theta+1/\alpha}$.
\end{itemize}
\end{assumption}

\begin{rem} \label{compare} Let us comment on Assumption~\ref{a:A}.
The  Lipschitz property guarantees that Eq.~\eqref{e:MaiSPDE} has a
unique solution, and~(A4) ensures that the solution is  strong Feller. The
condition $\sum_{k \geq 1} \frac{\beta^{\alpha}_k}{\gamma^{1-\alpha
\e}_k}<\infty$ in (A2) implies that the solution
to~\eqref{e:MaiSPDE}  evolves in linear subspace with compact embedding
into  $H$,  see Section~4. Note that in~\cite{PXZ10} it is only
required that (A2) holds for $\epsilon =0$ (i.e.,
 that $X^x_t \in H$, a.s.).
However, our present assumption with $\epsilon>0$ is  really
a mild assumption (compare also with Example~\ref{comp}).
\end{rem}

\subsection{Structural properties of solutions}
In this subsection we formulate the structural properties of
solutions in both finite  and infinite dimensions;
see Theorems~\ref{t:Sol} and~\ref{t:SolEU}. These structural properties will play an important role in proving the exponential ergodicity.
The proof of the next theorem is quite involved  and  is postponed
to Section~3.

\begin{thm} \label{t:Sol} Let $H = \R^n$.
Under Assumption~\ref{a:A2}, there exists a unique  strong solution $X^x_t$ for~\eqref{e:MaiSPDE}.
 The solutions $(X_t^x)_{x\in H}$ form a Markov process with  transition semigroup $P_t$,
  $$
P_t f(x)=\E[f(X^x_t)],  \;\;\; f \in B_b(H),
$$
  which is irreducible and such that
 there exists $C>0$ with
\begin{align} \label{grad12}
|P_t f(x) -  P_t f(y)|  \le
  \frac{C  \| f\|_0}{t^{1/\alpha} \wedge 1}
   |x-y|,\;\; x, y \in H,\; t>0,\; f \in B_b(H).
\end{align}
\end{thm}

%The formulation of the following finite dimensional %result
% is analogous  to the previous one. However here  a %more general noise $Z$ is considered.
 %which is symmetric non-degenerate and  % %$\alpha$-stable.

\vskip 2mm The following infinite-dimensional result is
analogous  to the previous one
 and is proved in~\cite{PZ09}. Note  that the noise $Z$
 considered here reduces in finite dimension
   to a  particular case of  the noise in Theorem~\ref{t:Sol}.

\begin{thm} \label{t:SolEU}
 Under
 Assumption~\ref{a:A},  there exists a unique mild solution $X^x_t$ for~\eqref{e:MaiSPDE},
\begin{\eqn} \label{e:MilSol}
X^x_t=e^{At}x+\int_0^t e^{A(t-s)} F(X^x_s)ds+\int_0^t e^{A(t-s)} dZ_s.
\end{\eqn}
The solutions $(X_t^x)_{x\in H}$ form a Markov process with the transition semigroup $P_t$.
  The process is  irreducible and there exists $C>0$ such that
\begin{align} \label{grad1}
|P_t f(x) -  P_t f(y)|  \le
  \frac{C  \| f\|_0}{t^{1/\theta} \wedge 1}
   |x-y|,\;\; x, y \in H,\; t>0,
\end{align}
where $\theta$ is given in (A4) of Assumption~\ref{a:A}.
\end{thm}
 \begin{rem}
 Note if $\dim H =  \infty$ then, in general,  trajectories of $(X_t^x)$ do not have a c\`adl\`ag modifications (see~\cite{BGIPPZ}).
\end{rem}
%\vskip 1mm

\subsection{Ergodic results for finite-dimensional equations}
\label{s2.2}

Let us denote by $(P_t)_{t\geq 0}$ the Markov semigroup associated with~\eqref{e:MaiSPDE} and by $(P^{*}_t)_{t \geq 0}$ the dual semigroup acting on $\mcl P(H)$.

 The main result for the finite-dimensional case is as follows:
  \begin{thm} \label{t:ErgFin}
Under Assumption~\ref{a:A2}, the system~\eqref{e:MaiSPDE} is ergodic and
exponentially mixing. More precisely, there exists $\mu \in \mcl P(H)$ such that, for  any $p \in (0, \alpha)$ and any measure $\nu \in \mcl P(H)$ with finite $p^{\rm th}$ moment, we have
\begin{equation} %\label{2.5}
\|P^{*}_t \nu- \mu\|_{\rm TV} \leq C e^{-ct}\left(1+\int_H |x|^{p}
\nu(dx)\right),
\end{equation}
where $C=C(p, \alpha, A, \|F\|_{0})$ and $c=c(p,
\alpha, A, \|F\|_{0})$.
\end{thm}

%Currently, there is a great interest in understanding pathwise uniqueness for SDEs when $F$ is not Lipschitz,
%see~\cite{FGP10,DPF10} for Gaussian noises and
%see the  references given in~\cite{DPF10, Pr10}.
%for $\alpha$-stable noises,
%Note that the deterministic equation~\eqref{e:MaiSPDE} with $Z_t
%\equiv 0$ is possibly not well-posed when $\theta<1$.

One can easily adapt our proof  to show that
the previous theorem is also true when~$(Z_t)$ is Gaussian.

\subsection{Ergodic results in the infinite-dimensional case}
The following theorem
describing the long-time behaviour of~$(X^x_t)$ is the main result
of the infinite-dimensional case.

\begin{thm}\label{t:MaiThm} Under Assumption~\ref{a:A},
the system~\eqref{e:MaiSPDE} is ergodic and exponentially mixing. More
precisely, there exists $\mu \in \mcl P(H)$ so that for any $p \in
(0, \alpha)$ and any measure $\nu \in \mcl P(H)$ with finite $p^{\rm
th}$ moment, we have
\begin{equation} \label{2.5}
\|P^{*}_t \nu- \mu\|_{\rm TV} \leq C e^{-ct}\left(1+\int_H |x|^{p}
\nu(dx)\right),
\end{equation}
where $C=C(p, \alpha, \theta, \beta,\gamma, \e, \|F\|_{0})$ and $c=c(p,
\alpha, \theta, \beta,\gamma, \e, \|F\|_{0})$
 with $\beta= (\beta_k)$, $\gamma = (\gamma_k)$.
\end{thm}

We will apply the above theorem in the following example which was considered in~\cite{PXZ10}. %Applying the above  theorem, we obtain

\begin{ex} \label{comp}
Consider the following semilinear parabolic SPDE in a bounded domain $D\subset\R^d$ with smooth boundary~$\p D$:
\Be
\label{e:SemLin}
\begin{cases}
dX(t,\xi)=[\Delta X(t,\xi)+F(X(t,\xi))]dt+dZ_t(\xi), \\
X(0,\xi)=x(\xi), \\
X(t,\xi)=0, \ \ \xi \in \p D,
\end{cases}
\Ee where $Z_t$ and $F$ are specified below. The Laplace
operator~$-\Delta$ with the Dirichlet boundary condition has a
discrete spectrum. We denote by~$\{e_k\}$ the set of its normalised
eigenfunctions and by~$\{\gamma_k\}$ the corresponding eigenvalues
written in increasing order and repeated according to multiplicity.
It is well known that $\gamma_k=C_d\,k^{2/d}(1+\e_k)$, where~$C_d$
is a constant depending on~$d$ and~$D$,  and~$\{\e_k\}$ is a
sequence going to zero as $k\to\infty$; see~\cite{agmon1965}.

We study the dynamics defined by~\eqref{e:SemLin} in the Hilbert space $H= L^2(D)$ with the orthonormal basis~$\{e_k\}$. Let us assume that~$Z= (Z_t)$ is  a cylindrical $\alpha$-stable noise written in the form
$$
Z_t=\sum_{k=1}^\infty \beta_k z_k(t)e_k,
$$
where $\{z_k(t)\}_k$ are i.i.d.\ symmetric $\alpha$-stable processes
with $\alpha \in (0,2)$. A straightforward calculation using the
above-mentioned asymptotics of~$\gamma_k$ shows that~(A2) and~(A4)
are satisfied simultaneously if and only if $2\alpha(\theta-\e)>d$.
Thus, if $d\le 3$, one can choose~$\alpha$, $\theta$, and~$\e$ for
which Assumption~\ref{a:A} holds, and we get the property of
exponential mixing in the total variation  norm for the dynamics
of~\eqref{e:SemLin}. This improves earlier results established in
Theorems 2.5 and 2.6 of~\cite{PXZ10} according to which strong
mixing holds under essentially the same hypotheses and exponential
mixing is true in the Kantorovich--Wasserstein metric if, in
addition, the norm~$\|F\|_0$ is sufficiently small.
\end{ex}

\subsection{Two approaches to exponential ergodicity} \label{s2.3} We shall prove the
exponential ergodicity results by two approaches. The
first one is by applying classical Harris'
 theorem and the other is by coupling argument.\vskip1mm

 We shall use the following Harris' theorem. For a
surprisingly short and nice proof, we refer to Hairer's lecture notes
\cite{Hai09}.
\begin{thm}[Harris] \label{t:Har}
\label{t:HaThm}
 Let $P_t$ be a Markov semigroup in the Polish space
$X$ such that there exists $T_0>0$ and $V : X \to \R_+ $ which
satisfies the following properties:
\begin{itemize}
\item[(i)]
there exists $\gamma <1$ and $K >0$ such that
 $P_{T_0} V(x) \le \gamma V(x) + K $, $x \in X$.
\item[(ii)]
for every $R>0$ there exists $\delta>0$ such that
$$
 \| P^*_{T_0} \delta_x  - P^*_{T_0} \delta_y \|_{TV} \le 1 - \delta,
$$
for all $x,y \in X$ such that $V(x) + V(y) \le R$.
\end{itemize}
Then there exist some $T>0$ and $\beta<1$ such that
$$
 \int_X (1+V(x))|P^*_T \mu  - P^*_T \nu|(dx) \le \beta
 \int_X (1+ V(x)) |\mu - \nu|(dx).
$$
 \end{thm}
The key point for Harris' theorem approach is to find a Lyapunov function $V$ and to check   conditions  (i) and (ii).

\vskip1mm
To sketch  the coupling approach, let us fix a large constant
$T>0$ and consider the restriction of the  Markov process~$(X_t^x)$,
$x \in H$, to the times proportional to~$T$. We denote by~$(Y_k)$
the resulting discrete-time Markov process, by~$\PP_x$ the
corresponding family of probability measures, and by~$P_k(x,\Gamma)$
the transition function.  The dissipativity of~$A$, the boundedness
of~$F$, and the non-degeneracy of~$Z$ imply that~$(Y_k)$ is
irreducible, and the first hitting time of any ball has a finite
exponential
 moment. Furthermore,  as  will follow from
Theorems~\ref{t:Sol} and~\ref{t:SolEU}, if the initial points
$x_1,x_2\in H$ are such that $|x_1-x_2|\le r$, with a sufficiently
small~$r>0$, then
\begin{equation} \label{1.2}
\|P_1(x_1,\cdot)-P_1(x_2,\cdot)\|_{\rm TV}\le \frac12.
\end{equation}
%where $\|\cdot\|_{\rm TV}$ stands for
%the total variation distance between two measures.

Now let $(Y_k^1,Y_k^2)$ be a homogeneous discrete-time Markov  process in the extended phase space~$H\times H$ such that the following properties hold
for the pair $(Y_1^1,Y_1^2)$ under the law~$\PP_{(x_1,x_2)}$ corresponding to the initial point~$(x_1,x_2)$:
\begin{itemize}
\item[(a)]
The laws of $Y_1^1$ and $Y_1^2$ coincide with $P_1(x_1,\cdot)$ and~$P_1(x_2,\cdot)$, respectively.
\item[(b)]
If $\max(|x_1|,|x_2|)>r $ and $x_1\ne x_2$, then the random variables $Y_1^1$ and~$Y_1^2$ are independent.
\item[(c)]
If $\max(|x_1|,|x_2|)\le r $ and $x_1\ne x_2$, then
\begin{equation*}
\PP_{(x_1,x_2)}\bigl\{Y_1^1\ne Y_1^2\bigr\} =\|P_1(x_1,\cdot)-P_1(x_2,\cdot)\|_{\rm TV}.
\end{equation*}
\item[(d)]
If $x_1=x_2$, then $Y_1^1=Y_1^2$ with probability~$1$.
\end{itemize}
Such a chain can be constructed with the help of maximal coupling of measures; see Section~\ref{s4}. Combining properties (a)--(d) with irreducibility
of~$(Y_k)$ and inequality~\eqref{1.2}, it is possible to prove that the stopping time $\rho=\min\{k\ge0:Y_k^1=Y_k^2\}$ is $\PP_{(x_1,x_2)}$-almost surely
finite and has a finite exponential moment. Moreover, it follows from~(d) that $Y_k^1=Y_k^2$ for $k\ge\rho$. We can thus write
\begin{equation} \label{1.3}
|P_k(x_1,\Gamma)-P_k(x_2,\Gamma)| =|\E_{(x_1,x_2)}(I_\Gamma(Y_k^1)-I_\Gamma(Y_k^2))| \le \PP_{(x_1,x_2)}\{\rho>k\},
\end{equation}
where $\Gamma\subset H$ is an arbitrary Borel subset and~$I_\Gamma$ stands for its indicator function. Since~$\rho$ has a finite exponential moment, the
right-hand side of~\eqref{1.3} can be estimated by $\const e^{-\gamma k}$. Taking the supremum over all Borel subsets~$\Gamma$, we conclude that the
total variation distance between~$P_k(x_1,\Gamma)$ and $P_k(x_2,\Gamma)$ goes to zero exponentially fast for any initial points~$x_1,x_2\in H$. This
implies the required uniqueness and exponential mixing.

\smallskip
In conclusion, let us note that, in the context of randomly forced PDE's, the coupling argument can be modified to cover the case of degenerate noises.
We refer the reader to~\cite{KS-2001,MY-2002,Shi04} for discrete-time random perturbations, to~\cite{Jon02,Hai02,KS-2002,Shi08,Od09} for a white noise,
to~\cite{nersesyan-2008} for a compound Poisson process, and to the book~\cite{KS10} for further references on this subject. We believe that a similar
approach can be developed in the case of dissipative PDE's driven by L\'evy noises.

\section{Proof of structural properties, $\dim H< \infty$} \label{s:F}
In this section, we concentrate on proving Theorem~\ref{t:Sol}, which can be done in the following steps.
\vskip 2mm \textit{Step 1. Existence and uniqueness.}
Since (with $X_t = X_t^x$)
\begin{align} \label{strong}
 X_t = x + \int_0^t A X_s ds + \int_0^t F (X_s)ds + Z_t,
 \end{align}
defining $v(t)= X_t - Z_t$, one can construct a c\`adl\`ag adapted solution, by working $\omega $ by $\omega$ and using a compacteness argument.

\vskip 1mm Uniqueness holds even in the limiting case $\alpha=1$.
 When $A=0$
 it  follows
 %,
 % even in the limiting case $\alpha=1$,
 directly
 from~\cite{Pr10}. In  the present case of $A \not =0$,
   since the  drift in~\cite{Pr10} was supposed to be
bounded and $x \mapsto Ax$ is an unbounded mapping,
to prove pathwise uniqueness one can proceed
into two different ways. First one can adapt the
computations in~\cite{Pr10} using
 a standard stopping
time argument. To this purpose, we only note that if
 $X_t$ is one solution starting from $x \in \R^n$ then
 formula in \cite[Lemma 4.2]{Pr10} continue to hold if
 $t$ is replaced by
 $t \wedge \tau_R$, $R>0,$ where
$$
 \tau_R = \inf\{t \ge 0; |X_t |\le R \}.
$$
Another method consists in introducing the process
$
Y_t=e^{-At} X_t$. Clearly $Y_t$ satisfies the following equation
\Be \label{e:YEqn} d Y_t=e^{-At} F(e^{At}Y_t)+e^{-At} dZ_t. \Ee According to~\cite{Pr10} with small
modifications (due to the fact that now the drift is bounded but also time-dependent), \eqref{e:YEqn} has a unique strong solution such that \Bes Y_t=x+\int_0^t e^{-As} F(e^{As}Y_s)ds+ \int_0^t e^{-As}Z_s, \Ees and this is
equivalent to~\eqref{strong}.

%\Bes X_t=e^{At}x+\int_0^t e^{A(t-s)} F(X_s)ds+\int_0^t e^{A(t-s)}
%dZ_s. \Ees

\noindent \textit{Step 2. Markov property.}  This follows from the uniqueness by standard considerations.

\vskip 1mm \noindent \textit {Step 3. Uniform strong Feller  estimate (\ref{grad1})}.

In order to adapt the method used in the proof of \cite[Theorem 5.7]{PZ09}, we need gradient estimates like
\begin{align}\label{gr}
 \| DR_t f\|_0 \le \frac{c}{t^{1/\alpha}} \| f\|_0, \;\; t \in (0,1], \; f \in B_b(H),
\end{align}
for the OU semigroup $R_t$ corresponding to $F=0$
 in~\eqref{strong}.

\begin{rem}
Some related estimates were obtained in a  recent paper~\cite{W10}
which however does not cover the present situation.
 We also mention~\cite{takeuchi} which contains a Bismut--Elworthy--Li
 formula for  jump diffusion semigroups (even without a Gaussian part). We cannot
 apply~\cite{takeuchi} since our L\'evy measure $\nu$  in general does not have a $C^1$-density with respect to the Lebesgue measure in $\R^n \setminus \{0\}$.
\end{rem}

The next result seems to be of independent interest.

\begin{theorem} Let $H = \R^n$.
Assume that $Z = (Z_t)$ is an $n$-dimensional symmetric non-degenerate $\alpha$-stable process, $\alpha \in (0,2)$. Consider any real $n \times n$ matrix
$A$. Then  gradient  estimates~\eqref{gr} holds for the
 OU semigroup $R_t$ associated with
$$
dX_t = A X_t dt +  dZ_t,\;\; X_0 =x.
$$
%i.e., for any $T>0$ there exists $C_T$ such that
%\begin{align} \label{grado}
% \| DR_t f  \|_0 \le \frac{C_T}{t^{1/\alpha }} \| %f\|_0,\;\; t \in
% (0,T],\;\; f \in B_b(H).
%\end{align}
\end{theorem}
\begin{proof} Let us fix $f \in B_b(H)$ and $t \in (0,T]$. It is
known (see, for instance, \cite{PZBU}) that
$$
R_t f (x) =  \int_{H } f (e^{tA} x + y ) p_t(y)(dy), $$
$$
 p_t(y) = \frac{1}{(2 \pi)^n}  \int_{H} e^{- i \langle y, h\rangle}
\exp { \Big(- \int_0 ^t \psi(  e^{s A^*}  h) ds \Big)} dh,
$$
where $\psi $ is the exponent (or symbol) of the L\'evy process $Z$ (see~\eqref{itol}).
We write
$$
R_t f (x)=\frac{1}{(2 \pi)^n} \int_{H } f (  z )
 \Big ( \int_{H} e^{- i
\lan z, h\ran} e^{  i \lan e^{tA^*}
 h, x\ran} e^{ - \int_0 ^t \psi(
e^{s A^*} h) ds } dh \Big) dz.
$$

\underline{\emph{(1)}}.
Recall the rescaling property
$$\psi (u s) = s^{\alpha} \psi (u), \;\;
 s \ge 0,
 $$ and $u \in
H$. The non-degeneracy assumption~\eqref{nondeg} implies that
there exists the  directional derivative along
 any fixed direction
$l \in H$, $|l|=1$
(cf. Section 3 in~\cite{Pr10}),
$$
 D_l R_tf(x) = \frac{i}{(2 \pi)^n} \int_{H } f (  z )
 \Big ( \int_{H} e^{- i
\lan z, h\ran} e^{  i \lan e^{tA^*}
 h, x\ran}  \,  \lan e^{tA^*}
 h, l\ran \, e^{ - \int_0 ^t \psi(
e^{s A^*} h) ds } dh \Big) dz.
$$
Let $e^{tA^* }h =k$. We have
$$
D_l R_tf(x) = \frac{i e^{- t \; tr(A)}}{(2 \pi)^n} \int_{H } f (  z
)
 \Big ( \int_{H} e^{- i
\lan z, e^{- tA^* }k\ran} e^{  i \lan k
 , x\ran}  \,  \lan k, l\ran \,
 e^{ - \int_0 ^t \psi( e^{(s-t) A^*} k) ds } dk \Big) dz
$$
$$
= \frac{i }{(2 \pi)^n} \int_{H } f (  e^{tA} \xi)
 \Big ( \int_{H} e^{- i
\lan \xi, k\ran} e^{  i \lan k
 , x\ran}  \,  \lan k, l\ran \,
 e^{ - \int_0 ^t \psi( e^{-r A^*} k) dr } dk \Big) d \xi
$$
$$
= \frac{i }{(2 \pi)^n} \int_{H } f (  e^{tA} \xi)
 \Big ( \int_{H} e^{  i \lan k
 , (x- \xi)\ran}  \,  \lan k, l\ran \,
 e^{ - \int_0 ^t \psi( e^{-r A^*} k) dr } dk \Big) d \xi.
$$
Let us introduce
$$
 \phi_t (v) = \frac{1 }{(2 \pi)^n}
 \int_{H} e^{  i \lan k
 , v \ran}  \,  \lan k, l\ran \,
 e^{ - \int_0 ^t \psi( e^{-r A^*} k) dr } dk.
$$
It is clear that we get
$$
 \| D_l R_t f  \|_0 \le \frac{C_1}{t^{1/\alpha }} \| f\|_0, \;\; t \in (0,1].
$$
(and so~\eqref{gr}) if we are able to prove that
\begin{align} \label{l1}
 \| \phi_t  \|_{L^1(H)} \le \frac{C_1}{t^{1/\alpha}},
  \;\; t \in (0,1],
\end{align}
where $L^1(H)= L^1(\R^n)$ with respect to the Lebesgue measure.
%\vskip 2mm

\underline{\emph{(2)}}.
Let us check~\eqref{l1}. Using the rescaling property, we have
$$
\phi_t (v)  = \frac{1 }{(2 \pi)^n}
 \int_{H} e^{  i \lan k
 , v \ran}  \,  \lan k, l\ran \,
 \exp \left\{ - \frac{1}{t} \int_0 ^t \psi( e^{-r A^*} t^{1/\alpha} k) dr   \right\} dk
$$
$$
= \frac{1 }{(2 \pi)^n  \, t^{n/\alpha}}
 \int_{H} \exp \left\{   i \lan \frac{h}{ t^{1/\alpha}}
 , v \ran \right\}  \,  \lan \frac{h}{ t^{1/\alpha}}, l\ran \,
 \exp \left\{  - \frac{1}{t} \int_0 ^t \psi( e^{-r A^*} h) dr   \right\} dh
$$
$$
= \frac{1}{ t^{1/\alpha}} \, \frac{1 }{(2 \pi)^n  \, t^{n/\alpha}}
 \int_{H} \exp \left\{   i \lan \frac{v}{ t^{1/\alpha}}
 , h \ran \right\}  \,  \lan {h} , l\ran \,
 \exp \left\{  - \frac{1}{t} \int_0 ^t \psi( e^{-r A^*} h) dr   \right\} dh.
$$
Since (with the change of variable: $v/t^{1/\alpha} =w$)
$$
 \int_H |\phi_t (v)| dv = \frac{1}{ t^{1/\alpha}}
  \frac{1 }{(2 \pi)^n}
\int_H \Big | \int_{H} e^{  i \lan w
 , h \ran}  \,  \lan {h} , l\ran \,
 \exp \left\{  - \frac{1}{t} \int_0 ^t \psi( e^{-r A^*} h) dr   \right\} dh
  \Big | dw,
$$
in order to prove~\eqref{l1} we need to show that
\begin{align} \label{l11}
 \| \varphi_t  \|_{L^1(H)} \le {C_1},
  \;\; t \in (0,1],
\end{align}
where
$$
 \varphi_t (w) = \frac{1 }{(2 \pi)^n}
 \int_{H} e^{ - i \lan w
 , h \ran}  \,  \lan {h} , l\ran \,
 \exp \left\{ - \frac{1}{t} \int_0 ^t \psi( e^{-r A^*} h) dr   \right\} dh.
$$

\underline{\emph{(3)}}.
Let us now show~\eqref{l11}. Write $\psi = \psi_1 + \psi_2$,
$$
 \psi_1 (u) =  \int_{ \{ |y| \le 1\} } \big(1-
  \cos \langle u,y \rangle
    \big ) \nu (dy),\;\;\;\;\; \psi_2  =  \psi - \psi_1,
$$
so that
\Bes
\begin{split}
 \varphi_t (w)
  = \frac{1 }{(2 \pi)^n}
 \int_{H} e^{ - i \lan w
 , h \ran}  \,  \lan {h} , l\ran \,
 e^{ - \frac{1}{t} \int_0 ^t \psi_1( e^{-r A^*} h) dr}
 e^{- \frac{1}{t} \int_0 ^t \psi_2( e^{-r A^*} h) dr}
   dh.
\end{split}
\Ees
Now consider the random variable
$$
Y_t = \frac{1}{t^{1/\alpha}} \int_0^t e^{-(t-s) A} dZ^2_s, \;\; t \in (0,1],
$$
where $Z^2 = (Z^2_t)$ is a L\'evy process having exponent
$\psi_2$. It is easy to check that its law $\mu_t$ has characteristic
function $e^{ - \frac{1}{t} \int_0 ^t \psi_2( e^{-r A^*} h) dr   }$, i.e.,$$
\hat \mu_t (h) = \exp \left\{ - \frac{1}{t} \int_0^t \psi_2( e^{-r A^{*}} h) dr   \right\}, \;\; h \in H.
$$
Now suppose that there exists $g_t \in L^1 (H)$,
$t \in (0,1],$ such that
 \begin{equation} \label{gt}
\hat g_t (h)=\langle h, l \rangle
 \exp\left\{-\frac 1t \int_0^t \psi_1( e^{-r A^{*}} h) dr\right\}.
 \end{equation}
Then, by well known properties of the Fourier transfom (see Proposition 2.5 in~\cite{sato}) we would get
$$
 \hat{g_t} \cdot \hat \mu_t =  \widehat {g_t * \mu_t}
$$
and, using the Fourier inversion formula,
$$
\varphi_t (w) = (g_t * \mu_t) (w),
$$
so that  $\| \varphi_t \|_{L^1} \le \| g_t\|_{L^1}$, $t \in (0,1]$. Thus to prove~\eqref{l11} and get the assertion, it remains to show that~\eqref{gt} holds and moreover that
\begin{align} \label{l111}
 \| g_t  \|_{L^1(H)} \le {C_1},
  \;\; t \in (0,1].
\end{align}
%\vskip 2mm

\underline{\emph{(4)}.} Now we show~\eqref{gt} and~\eqref{l111}.
Note that
\begin{equation*}
\begin{split}
& \exp \left\{ - \frac{1}{t} \int_0 ^t \psi_1( e^{-r A^*} h) dr  \right\}
= \exp \left\{ - \frac{1}{t} \int_0 ^t
 dr\,  \int_{ \{ |y| \le 1\} } \big( 1 - \cos (\langle
   e^{-r A^*} h, y \rangle)
    \big ) \nu (dy) \right\}\\
& = \exp \left\{ - \frac{1}{t} \int_0 ^t \psi( e^{-r A^*} h) dr \right\}
\;   \exp \left\{   \frac{1}{t} \int_0 ^t
  dr\, \int_{ \{ |y| > 1\} } \big( 1 - \cos (\langle
   e^{-r A^*} h, y \rangle)    \big ) \nu (dy)\right\}\\
& \ \ \ \ \ \ \ \
\le \exp \left\{ 2 \nu (\{ |y| >1 \})\right\} \, \exp \left\{ -  \frac{C_{\alpha}}{t} \int_0 ^t |e^{-r A^*} h|^{\alpha} dr \right\}.
\end{split}
\end{equation*}
Since $|h| \le c_2 |e^{-r A^*} h |$, $h \in H$, $r \in [0,T]$, it follows that
\begin{align} \label{esti}
\exp \left\{ - \frac{1}{t} \int_0 ^t
\psi_1( e^{-r A^*} h) dr    \right\} \le c_1 e^{ - c_3 |h|^{\alpha}},\;\; h \in H, \; t \in (0,1].
\end{align}
 We find easily that  $\psi_1 \in C^{\infty}(H)$
 and so, using also~\eqref{esti}
 we deduce that the mapping $h \mapsto
 \lan {h} , l\ran \,
 e^{ - \frac{1}{t} \int_0 ^t \psi_1( e^{-r A^*} h) dr   }$
  is in the Schwartz space ${\cal S} (H)$, for any $t \in (0,1]$. It follows  that there exists $g_t \in {\cal S} (H)$ such that~\eqref{gt}  holds. By the inversion formula,
  $$
   g_t(w)  = \frac{1 }{(2 \pi)^n}
 \int_{H} e^{ - i \lan w
 , h \ran}  \,  \lan {h} , l\ran \,
 \exp \left\{ - \frac{1}{t} \int_0 ^t \psi_1( e^{-r A^*} h) dr   \right\}
  \,
   dh,\;\; w \in H.
  $$

  Now we show~\eqref{l111}, by proving that
  for any multiindex $\beta = (\beta_1, \ldots , \beta_n)
  \in {\mathbb Z}_+^n$, there exists $c_T$ such that
  (with $w^{\beta}:=w_1^{\beta_1} \cdots w_n^{\beta_n}$)
 \begin{align}  \label{multi}
  \sup_{w \in H}|w^{\beta} g_t(w)| = c_1
   < \infty,\;\; t \in ]0,1]
 \end{align}
 (note that the
  constant $c_1$ is independent of $t$).
 Indeed once~\eqref{multi} is proved then
 $$
  \| g_t \|_{L^1} \le c_1' \int_H \frac{1}{1 + |w|^{2n}} dw = c^{''}_1< \infty.
 $$
 We will check~\eqref{multi} only for $w^{\beta} = w_j$,
 i.e. $\beta = (0, \ldots ,1, \ldots, 0)$ with 1 in
 the $j$-th position. The proof
  in the general case  is similar.

 We have, integrating by parts and using estimate
~\eqref{esti},
 $$
  w_j \,  g_t (w) = \frac{1 }{(2 \pi)^n}
 \int_{H} w_j e^{ - i \lan w
 , h \ran}  \,  \lan {h} , l\ran \,
 \exp \left\{  - \frac{1}{t} \int_0 ^t \psi_1( e^{-r A^*} h) dr   \right\}
  \,
   dh
 $$
  $$
= \frac{i }{(2 \pi)^n}
 \int_{H} \partial_{h_j} \big( e^{ - i \lan w
 , h \ran} \big)  \,  \lan {h} , l\ran \,
\exp \left\{ - \frac{1}{t} \int_0 ^t \psi_1( e^{-r A^*} h) dr   \right\}
  \,
   dh
$$
$$
= - \frac{i }{(2 \pi)^n}
\int_{H} e^{ - i \lan w
 , h \ran}   \, l_j  \,
 \exp \left\{ - \frac{1}{t} \int_0 ^t \psi_1( e^{-r A^*} h) dr   \right\}
  \,
   dh
$$
$$
-\frac{i }{(2 \pi)^n}   \int_{H} e^{ - i \lan w
 , h \ran}  \,  \lan {h} , l\ran \,
 e^{ - \frac{1}{t} \int_0 ^t \psi_1( e^{-r A^*} h) dr   }
  \, \Big( - \frac{1}{t} \int_0 ^t \langle D\psi_1
  ( e^{-r A^*} h),
      e^{-r A^*} e_j \rangle dr
  \Big)
   dh.
$$
Using~\eqref{esti} and the fact the $|D\psi_1(u)| \le c_5 |u|$,
$u \in H$, get easily that
$$\sup_{w \in H}|w_j \, g_t(w)| = c_1
   < \infty,\;\; t \in ]0,1].
$$
The proof is complete.
\end{proof}

 \vskip 1mm \noindent \textit{Step 4. Irreducibility.}
  We cannot
 argue as  in the proof
of \cite[Theorem 5.3]{PZ09} since the drift $F$ is only H\"older
continuous. Note, however, that if we prove that the
Ornstein--Uhlenbeck process $Z_A = (Z_A(t)),$
\begin{align} \label{ztt}
Z_A(t)= \int_0^t e^{A(t-s)} dZ_s
\end{align}
(starting at $x=0$), is  irreducible  then we can obtain
irreducibility for the solution $X^x$ using  the following quite
general result of independent interest.

\begin{proposition} \label{t:IrrFin} Assume that for each $t>0$ the support of  $Z_A(t)$ is the whole space. Then the process $(X_t^x)$ is
irreducible, for any $x \in H$.
\end{proposition}
\begin{proof} Fix $t>0$,  $a>0$ and let $r>0$ be any positive number. Then
$$
X_{t+a} =e^{Aa}X_t +\int_t^{t+a} e^{A(t+a-s)} F(X_s)ds+\int_t^{t+a} e^{A(t+a-s)} dZ_s.
$$
Let $z$ be any element in the support of the distribution of the random variable $e^{Aa}X_t$. Then, by the very definition, the event
$$
B= \{ |e^{Aa}X_t -z| < {r}/{3}  \}
$$
is of positive probability. Since $||F||_0<\infty$, there exists $c>0$ such that for each $t\geq 0$ and for each positive $b$ with probability $1$
$$
 \Big| \int_t^{t+b} e^{A(t+b-s)} F(X_s)ds \Big| \le c b,
$$
In particular, the above inequality holds for $b=a$. Let us fix $x$ and $y$ in $H$. Then
$$
X_{t+a} -y =(e^{Aa}X_t - z) + \int_t^{t+a} e^{A(t+a-s)} F(X_s)ds+
\Big(\int_t^{t+a} e^{A(t+a-s)} dZ_s - y +z \Big).
$$
Define the event
$$
C = \Big \{  \Big| y-z -\int_t^{t+a} e^{A(t+a-s)} dZ_s
 \Big| < r/3 \Big \},
$$
which, by assumption, is of positive probability. The events $B$ and $C$ are independent and therefore the probability of $B\cap C$ is positive. On
this event, and thus with positive probability,  we have the estimate:
$$
|X_{t+a} -y| \leq \frac {r}{3} + ca + \frac{ r}{3}.
$$
Starting from number $a$ such that $ca <  r/3$ we have with positive probability
$$
|X_{t+a} -y| \leq r .
$$
To finish the proof we should replace $t+a$ and $t$ with $t$ and
$t-a$.
\end{proof}

 By the previous result, we know that the proof of Step 4 is complete once the following theorem has been proved.

 \begin{theorem} \label{irred}
Let $H = \R^n$. Assume that $Z = (Z_t)$ is an $n$-dimensional symmetric non-degenerate $\alpha$-stable process, $\alpha \in (0,2)$. Consider any real $n \times n$ matrix
$A$. Then  the Ornstein--Uhlenbeck process $X(t)= Z_A(t)$ (given in
\eqref{ztt} and starting at $x=0$)
 is irreducible i.e., for any $t>0$ the support of the distribution of $X(t)$ is $H.$
\end{theorem}

\begin{proof}  By the non-degenerate assumption~\eqref{spec} there exists $n$ points
$a_1, \ldots, a_n \in S$ such that $a_k \in supp(\mu)$ for $1 \leq k \leq n$ and $span\{a_1, \ldots, a_n\}=\R^n$.
Since $\mu$ is symmetric, $-a_1, \ldots, -a_n \in supp(\mu)$.
It is clear that for any $\e>0$, $\mu(B_s(\pm a_k,\e))>0$
where $B_s(a_k, \e)=\{y \in S; |y-a_k|<\e\}$.

For each $k$, let us now consider the affines ${\mcl F}_{k,+}:=\{r a_k, r>1\}$ and ${\mcl F}_{k,-}:=\{-r a_k, r>1\}$.
For any point $y_k \in \{r a_k, -\infty<r<\infty\}$, there exist $y_{k,+} \in {\mcl F_{k,+}}$ and $ y_{k,-} \in {\mcl F_{k,-}}$ such
that $y_k=y_{k,+}+y_{k,-}$. Define $\mcl F^{+}_{k,\e}:= \{ (x,r)\, :\, x \in B_s(a_k,\e),\,  r>1 \}$,
$\mcl F^{-}_{k,\e}= \{ (x,r)\, :\, x \in B_s(-a_k,\e),\,  r>1 \}$, %where $B_s(a_k,\e)=\{x \in S;|x-a_k|<\e\}$ for $1 %\leq k \leq n$.
Take $\e>0$ small enough to make
$\mcl F^{\pm}_{i,\e} \cap \mcl F^{\pm}_{j,\e}=\emptyset$ for $i \neq j$ and
$\mcl F^{+}_{i,\e} \cap \mcl F^{-}_{i,\e}=\emptyset$ for each $i$.

 Decompose $\nu$ as the sum of two measures $\nu_1$,  $\nu_2$ such that
$$
\nu = \nu_1 +\nu_2 ,
$$
and one of the measures, say $\nu_1=\nu 1_{(\cup_{k=1}^n \mcl F^{+}_{k,\e})
 \cup (\cup_{k=1}^n \mcl F^{-}_{k,\e})}$, is finite. We can
assume that the process $Z$ is the sum of two independent
L\'evy processes $Z^1$and $Z^2$, with the L\'evy measures $\nu_1$ and $\nu_2$ respectively.
Note that
$$X^1(t): = \int_0^t e^{A(t-s)} dZ_s^1,\;\; t \ge 0,
$$
is a compound Poisson process. Since
$supp(\mu_1) \subset supp(\mu_1 *\mu_2)$ for any two measures $\mu_1$ and $\mu_2$, it is
 enough to prove the irreducibility of $X^1.$

Let us fix $t>0$, $y\in H$
and $r>0$. It is enough to show that
$$
\PP (|X^1 (t) - y|< r ) >0 .
$$
Let $M$ be a number such that for all $s\in (0,1)$:
$$
|e^{As}z| \leq M |z|,\,\,\,\,|(e^{As}-I)z|\leq M s |z|,\,\,z\in H.
$$
Write $y=\sum_{k=1}^n y_k a_k$ where $y_1, \ldots,y_n\in \R$, for each $k$ we have two points $y_{k,+} \in {\mcl F}_{k,+}$ and $y_{k,-} \in {\mcl F}_{k,-}$
and positive number $\delta <1$ such that:
$$
y_{k,+}+y_{k,-}=y_k a_k,\,\, \;\; \delta M \left(|y_{k,+}|+|y_{k,-}|\right)<{\frac{r}{2n}}.
$$
Choose $\e>0$ sufficiently small, the probability that the process $Z^1$ will perform exactly  $2n$ jumps $\xi_{1,-} \in \mcl F^{-}_{1,\e}, \xi_{1,+} \in \mcl F^{+}_{1,\e}, \ldots, \xi_{n,-} \in \mcl F^{-}_{n,\e},\xi_{n,+} \in \mcl F^{+}_{n,\e}$  before $t$ at moments $\tau_{1,-} <\tau_{1,+}<\tau_{2,-} <\tau_{2,+}<\ldots<\tau_{n,-} <\tau_{n,+} <t$ such
that
$$
\tau_{1,-} >t-\delta,\,\,\,\,\,\, |\xi_{k,-} - y_{k,-}|<{\frac{r}{4nM}}, \ \ldots , \ |\xi_{k,+} - y_{k,+} |<{\frac{r}{4nM}}, \ \ \ k=1,\cdots,n,
$$
is positive. Therefore, at least with the same probability, the following relations hold:
\Bes
\begin{split}
& \ \ \left|\int_0^{t} e^{(t-s) A} dZ^1_s-y\right| \\
&=\left|\sum_{j=1}^{n}e^{A(t-\tau_{j,-})}\xi_{j,-}+e^{A(t-\tau_{j,+})}\xi_{j,+} - y\right|\\
&=\left|\sum_{j=1}^{n}e^{A(t-\tau_{j,-})}(\xi_{j,-}-y_{j,-})+e^{A(t-\tau_{j,+})}(\xi_{j,+} - y_{j,+})\right|\\
& \ \ \ +\left|\sum_{j=1}^{n}(e^{A(t-\tau_{j,-})}-I)y_{j,-}+(e^{A(t-\tau_{j,+})}-I)y_{j,+}\right| \\
& \leq \sum_{j=1}^n M \left(|\xi_{j,-}-y_{j,-}|+|\xi_{j,+} - y_{j,+}|\right)+\sum_{j=1}^n \delta M \left(|y_{j,-}|+|y_{j,+}|\right) <r.
\end{split}
\Ees
This finishes  the proof.
\end{proof}
The proof of Theorem~\ref{t:Sol} is now complete.

\section{Estimates of the solution, $\dim H = \infty$}
\label{s3} This section contains some preparation for the proof of
Theorem~\ref{t:MaiThm}, giving some estimates for the solution
\eqref{e:MilSol}. Recall that the Ornstein--Uhlenbeck process is
defined by \Be \label{e:OUAlp} Z_A(t)=\int_0^t e^{A(t-s)}
dZ_s=\sum_{k \geq 1} Z_{A,k}(t) e_k\Ee where
\begin{\eqn*}
 Z_{A,k}(t)=\int_0^t e^{-\gamma_k(t-s)},
\beta_k dz_k(s).
\end{\eqn*}
For any $\e \geq 0$, define
$$
H^\e=\left\{x=\sum_{k\geq1} x_ke_k\in H: \sum_{k \geq 1} \gamma^{2 \e}_k |x_k|^2<\infty \right\}.
$$
Note that~$H^\e$ coincides with the domain of~${(-A)^{\e}}$ and that $H^0=H$. Denote further by $|\cdot|_\e$ the norm of $H^\e$. For $x \in H^\e$ and
$R>0$, we denote by~$B_\e (x,R)$ the closed ball in~$H^\e$ of radius~$R$ centered  at~$x$. We shall write $B_\e (R):=B_\e(0,R)$ and $B(x,R):=B_0(x,R)$.

\BL \label{ser} The following assertions hold:

(i) $Z_A(t) \in H^{\e} \ a.s.$ for all $t>0$.

(ii) For any $p
\in (0, \alpha)$, we have
\begin{\eqn} \label{e:ZAEst} \E |Z_A(t)|^p_{\e}  \leq C
\left(\sum_{k \geq 1} |\beta_k|^\alpha \
 \frac{1-e^{-\alpha \gamma_k t}}{\alpha \gamma^{1-\alpha \e}_k}\right)^{\frac p
 \alpha},
\end{\eqn}
where $C=C(\alpha, p)>0$.
 \EL
\Bp (i). By (4.7) in~\cite{PZ09} we have $$\E[e^{i \lambda
Z_{A,k}(t)}]= e^{-|\lambda|^\alpha c^\alpha_k(t)},
$$ where $c_k(t)=\beta_k
\left(\frac{1-e^{-\alpha \gamma_k t}}{\alpha
\gamma_k}\right)^{1/\alpha}$. Hence, $Z_{A,k}(t)$ has the same
distribution as $c_k(t) \xi_k$ for all $k \geq 1$ where
$\{\xi_k\}_{k \geq 1}$ are i.i.d. with $\E[e^{i \lambda
\xi_1}]=e^{-|\lambda|^\alpha}$. We shall use Proposition 3.3 in
\cite{PZ09}, which claims that
$$(q_k \xi_k)_{k \geq 1} \in l^2 \ \ a.s.  \Longleftrightarrow \sum_{k \geq 1} |q_k|^\alpha<\infty,$$
where $q_k \in \R$ for all $k$. From this it is easy to check that
\begin{\eqn*}
 \sum_{k \geq 1}
\left(\gamma_k\right)^{2\e}\left[c_k(t) \xi_k\right]^{2} <\infty \ \  a.s. \Longleftrightarrow \sum_{k \geq 1}
\frac{\beta^{\alpha}_k}{\gamma^{1-\alpha \e}_k}<\infty.
\end{\eqn*}
Since $Z_A(t)$ has the same distribution as $(c_k(t) \xi_k)_{k \geq
1}$, (i) is clearly true. \vskip 2mm

(ii). We follow the argument in the proof of \cite[Theorem 4.4]{PZ09}. Take a Rademacher
sequence $\{r_k\}_{k \geq 1}$ in a new probability space
$(\Omega^{'},\mcl F^{'},\PP^{'})$, i.e. $\{r_k\}_{k \geq 1}$ are
i.i.d. with $\PP\{r_k=1\}=\PP\{r_k=-1\}=\frac 12$. Recall the following
Khintchine inequality: for any $p>0$, there exists some $C(p)>0$
such that for arbitrary real sequence $\{h_k\}_{k \geq 1}$,
$$\left(\sum_{k \geq 1} h^2_k\right)^{1/2} \leq C(p) \left(\E^{'} \left|\sum_{k \geq 1} r_k h_k\right|^p\right)^{1/p}.$$
By this inequality, one has
\begin{\eqn} \label{e:EZAtp}
\begin{split}
\E|Z_A(t)|^p_{\e}&=\E \left(\sum_{k \geq 1} \gamma^{2\e}_k
|Z_{A,k}(t)|^2\right)^{p/2} \leq C \E \E^{'}\left|\sum_{k \geq 1}
r_k \gamma^\e_kZ_{A,k}(t)\right|^p \\
&=C\E^{'}\E \left|\sum_{k \geq 1} r_k \gamma^\e_k
Z_{A,k}(t)\right|^p,
\end{split}
\end{\eqn}
where $C=C^p(p)$. In view of the equality $|r_k|=1$ and formula~(4.7) of~\cite{PZ09}, for any $\lambda \in \R$ one has
\begin{\eqn*}
\begin{split}
\E\exp \left\{i \lambda \sum_{k \geq 1} r_k \gamma^\e_k
Z_{A,k}(t)\right\}&=\exp\left\{-|\lambda|^\alpha \sum_{k \geq 1}
|\beta_k|^\alpha  \gamma^{\e\alpha}_k \int_0^t e^{-\alpha \gamma_k
(t-s)} ds \right\} \\
&=\exp\left\{-|\lambda|^\alpha \sum_{k \geq 1} \gamma^{\e\alpha}_k
c^\alpha_k(t) \right\}.
\end{split}
\end{\eqn*}
Now we use (3.2) in~\cite{PZ09}: if $X$ is a symmetric random
variable satisfying $\E \left[e^{i\lambda
X}\right]=e^{-\sigma^{\alpha} |\lambda|^\alpha}$ for some $\alpha
\in (0,2)$ and any $\lambda \in \R$, then $\E|X|^p=C(\alpha,p)
\sigma^p$ for all $p \in (0,\alpha)$.
 Since $\sum_{k \geq 1} \gamma^{\e\alpha}_k c^\alpha_k(t)<\infty$, it
is clear to see
\begin{\eqn*}
\E \left|\sum_{k \geq 1} r_k \gamma^\e_k
Z_{A,k}(t)\right|^p=C(\alpha,p)\left(\sum_{k \geq 1}
|\beta_k|^\alpha \
 \frac{1-e^{-\alpha \gamma_k t}}{\alpha \gamma^{1-\alpha \e}_k}\right)^{\frac p
 \alpha},
\end{\eqn*}
from which and~\eqref{e:EZAtp} we get~\eqref{e:ZAEst}.
 \end{proof}

\begin{lem} \label{l:XtHe}
 Let $(X^x_t)$ be the solution to Eq.~\eqref{e:MaiSPDE} with $x
\in H^\e$. For any $p \in (0,\alpha)$, there exist some constants
$C_1=C_1(p)>0$ and $C_2=C_2(p, \e, \gamma, \beta, \|F\|_0)>1$ such
that

\begin{\eqn} \label{e:XtHe}
\E |X^x_t|^p_{\e} \leq C_1 e^{-p \gamma_1 t} |x|^p_{\e}+C_2, \ \ \ \
\forall \ t>0,
\end{\eqn}
\\
where $C_1(p) \leq 1$ for $p \in (0,1]$ and $C_1(p)=3^{p-1}$
otherwise.
\end{lem}
\vskip 1mm
\begin{proof}
By~\eqref{e:MilSol}, we have
\begin{\eqn*}
X_t=e^{At} x+\int_0^t e^{A(t-s)} F(X_s)ds+Z_A(t).
\end{\eqn*}
It is easy to see
$$|e^{At} x|_{\e} \leq e^{-\gamma_1 t} |x|_{\e}.$$
 By the easy inequality $|(-A)^{\sigma} e^{At}|_{L(H)} \leq
C(\sigma) t^{-\sigma}$, $t \ge 0$, $\sigma >0$, one has

\begin{\eqn*}
\begin{split}
\left|\int_0^t e^{A(t-s)} F(X_s)ds \right|_{\e} & \leq \int_0^t
|{(-A)^{\e}} e^{A(t-s)/2}|_{L(H)} |e^{A(t-s)/2}
F(X_s)| ds \\
& \leq C(\e) \int_0^t (t-s)^{-\e} e^{-\gamma_1 (t-s)/2} ds
\|F\|_{0} \\
& \leq C(\e, \gamma_1) \|F\|_{0}.
\end{split}
\end{\eqn*}
\\
 for all $t>0, x \in H$ and $\omega \in \Omega$. Furthermore, from
\eqref{e:ZAEst},
\begin{\eqn*}  \E |Z_A(t)|^p_{\e}  \leq C(p,\alpha,
\beta, \gamma, \e), \ \ \ \ \forall \  p \in (0, \alpha).
\end{\eqn*}
 Now we use the
following trivial inequality: for any $a,b,c \geq 0$,
\begin{\eqn*}
\begin{split}
&(a+b+c)^p \leq \left(a^p+b^p+c^p\right),  \ \ \ p \leq 1; \\
&(a+b+c)^p \leq 3^{p-1}\left(a^p+b^p+c^p\right),  \ \ \ p>1.
\end{split}
\end{\eqn*}
 Combining the above
three estimates and the inequality, we can easily see that
\eqref{e:XtHe} is true.
\end{proof}

\begin{lem} \label{vai}
Let $(X^x_t)$ be the solution to Eq.~\eqref{e:MaiSPDE}. For any $p
\in (0, \alpha)$, we have

\begin{\eqn} \label{e:XtHe1}
\E |X^x_t|^p_{\e} \leq C\left(t^{-\e p}|x|^p+t^{p- \e
p}\|F\|^p_{0}+1\right)
\end{\eqn}
\\
for all $t>0$, where $C=C(p,\alpha, \beta, \gamma,\e)$.
\end{lem}

\begin{proof}
By~\eqref{e:MilSol} and~\eqref{e:ZAEst}, we have

\begin{\eqn*}
\begin{split}
\E|X^x_t|^p_{\e} & \leq C_1\left[|A^\e e^{At}x|^p+\E \left(\int_0^t
|A^\e e^{A(t-s)}|_{L(H)}|F(X^x_s)| ds\right)^p+\E |Z_A(t)|^p_{\e}\right]\\
& \leq C_2 \left[t^{-\e p} |x|^p+\left(\int_0^t
(t-s)^{-\e} ds\right)^p \|F\|^p_{0}+1\right] \\
& \leq C_3 \left(t^{-\e p}|x|^p+t^{p- \e p}\|F\|^p_{0}+1\right),
\end{split}
\end{\eqn*}
\\
where $C_1=C_1(p)$ and $C_i=C_i(p,\alpha, \beta, \gamma,\e)$
($i=2,3$).
\end{proof}

\section{Proof of  Theorem~\ref{t:MaiThm} by Harris' approach,
$\dim H =\infty$} \label{s:Har}

Let us split the proof into the following three steps.

\medskip
{\it Step~1}. The existence of an invariant measure was established
in~\cite{PXZ10}.
 Let us prove that any invariant measure~$\mu$ has
finite $p^{\text{th}}$ moment ($p < \alpha$):
\begin{\eqn} \label{e:MuExp} \mmmm_p(\mu):=\int_H |x|^p \mu(dx)<\infty\quad\mbox{for any $p \in (0,\alpha)$}.
\end{\eqn}
Indeed, by~\eqref{e:MilSol} and the trivial inequality
$$
(a+b) \wedge c \leq a \wedge c+b \wedge c, \quad a, b, c \in \R^{+},
$$
for all $t>0$ and $n \in \N$, we have
\begin{\eqn*}
|X^x_t|^p \wedge n \leq  \left[\big(C_p e^{-p\gamma_1 t}
 |x|^p\big) \wedge n+C_p\left|\int_0^t e^{A(t-s)} F(X_s)ds\right|^p+C_p|Z_A(t)|^p \right].
\end{\eqn*}
Using a similar calculation as in Lemma~\ref{l:XtHe}, we obtain
\begin{\eqn*}
\E \big(|X^x_t|^p \wedge n\big) \leq  \big(C_pe^{-p\gamma_1 t} |x|^p \big)\wedge n+C,
\end{\eqn*}
where $C=C(\alpha, \beta, \gamma, p, \|F\|_{0})$.
Integrating this inequality against~$\mu(dx)$, we get
\begin{\eqn*}
\mu(|x|^p \wedge n) \leq  \mu \left[\big(C_pe^{-p\gamma_1 t}|x|^p\big) \wedge n\right]+C.
\end{\eqn*}
Passing to the limit first as $t \rightarrow \infty$ and then as $n \uparrow \infty$, we complete the proof of~\eqref{e:MuExp}.

\medskip

{\it Step~2}.
To prove the uniqueness of an invariant measure and inequality~\eqref{2.5}, it suffices to show that
\begin{equation} \label{5.2}
\|P_{kT}(x_1,\cdot)-P_{kT}(x_2,\cdot)\|_{\rm TV}
\le C\,(1+|x_1|^p+|x_2|^p)e^{-ckT},
\quad x_1,x_2\in H,
\end{equation}
where $C$ and~$c$ are positive constants not depending on $x_1$,
$x_2$, and~$k$. Indeed, if~\eqref{5.2} is established, then for any
measures $\nu_1,\nu_2\in {\mathcal P}(H)$ with finite
$p^{\text{th}}$ moment we derive
\begin{equation} \label{5.3}
\|P_{kT}^*\nu_1-P_{kT}^*\nu_2\|_{\rm TV}
\le C\,\bigl(1+\mmmm_p(\nu_1)+\mmmm_p(\nu_2)\bigr)e^{-c kT},
 \quad k\in \N.
\end{equation}
This implies, in particular, that an invariant measure is unique. Moreover, writing any $t\ge0$ in the form $t=kT+s$ with $0\le s<T$ and using inequalities~\eqref{5.3} and~\eqref{e:XtHe}, we obtain
\begin{align*}
\|P_{t}^*\nu_1-P_{t}^*\nu_2\|_{\rm TV}
&=\|P_{kT}^*(P_s^*\nu_1)-P_{kT}^*(P_s^*\nu_2)\|_{\rm TV}\\
&\le C\,\bigl(1+\mmmm_p(P_s^*\nu_1)+\mmmm_p(P_s^*\nu_2)\bigr)e^{-c kT}\\
&\le C_1\bigl(1+\mmmm_p(\nu_1)+
\mmmm_p(\nu_2)\bigr)e^{-c t}.
\end{align*}
This estimate readily implies the required inequality~\eqref{2.5}.

\vskip 1mm

Note that~\eqref{5.2} holds if we are able to  apply   Theorem~\ref{t:HaThm}
to equation~\eqref{e:MaiSPDE} with $V(x)=|x|^p$ and $p \in (0,
\alpha)$. Indeed, once this is done,
 we obtain that there exists $T>0$ such that
  \Bes
\begin{split}
\|P_{kT}(x_1,\cdot)-P_{kT}(x_2,\cdot) \|_{TV} & \leq \int_H (1+V(x))|P^*_{kT} \delta_{x_1}  - P^*_{kT} \delta_{x_2}|(dx)  \\
 & \leq \beta^k \int_H (1+ V(x)) |\delta_{x_1} - \delta_{x_2}|(dx) \\
& \leq 2\beta^k \bigl(1+|x_1|^p+|x_2|^p\bigr),\;\;\; k \ge 1.
\end{split}
\Ees
This immediately implies~\eqref{5.2}.
\medskip

{\it Step~3}.
  It remains  to check the conditions (i) and (ii) in Theorem~\ref{t:Har}. Choosing
$V(x)=|x|^p$ with $p \in (0, \alpha)$ and applying Lemma~\ref{l:XtHe} with
$\e=0$ and $T_0>\frac{\log(1+C_1)}{p \gamma_1}$, one immediately get (i).

\vskip 1mm
To prove (ii), we shall use the following lemma proved in~\cite{PZ09}.

\begin{lem}  [Theorem 5.4, \cite{PZ09}] \label{p:Irr}
Let $(X^x_t)$ be the solution to Eq.~\eqref{e:MaiSPDE}. Then
$(X^x_t)$ is irreducible on $H$, i.e., for any $t>0$ and $B(y,r)$
with arbitrary $y \in H$ and $r>0$, we have
\begin{\eqn}
\PP \left (X^x_t \in B(y,r)\right )>0.
\end{\eqn}
\end{lem}
\medskip

Let $x$ and $y$ satisfy $|x|^p+|y|^p \le R$. By Lemma~\ref{vai} we know that, for any fixed $T_0>0$,
$$
\E [ |X_{T_0}^x|^p_{\epsilon} ] + \E [ |X_{T_0}^y|^p_{\epsilon} ] \le C(|x|^p + |y|^p+1) \le C_1.
 $$
It follows that there exists some $R_1>0$ such that
$$\PP\left(|X^{x}_{T_0}|_\e \leq R_1\right)>1/2, \ \ \PP\left(|X^{y}_{T_0}|_\e \leq R_1\right)>1/2.$$

\noindent Since $\gamma_k \rightarrow \infty$,
$B_\e(M)$ is compact in $H$. By Lemma~\ref{p:Irr}, for any $r>0$ we have some $\delta(r)>0$ such
that
\Be \label{e:ComPro}
\inf_{x \in B_\e(R_1)} \PP\left(X^x_{T_0} \in B(r)\right) \ge 2\delta.
\Ee
By Markov property and the above three inequalities,
\Bes
\PP\left(X^x_{2T_0} \in B(r)\right)>\delta, \ \ \PP\left(X^y_{2T_0} \in B(r)\right)>\delta.
\Ees
Without loss of generality, in the next computations we  assume that $X^x_{t}$ and $X^y_{t}$ are independent (this is true
if the driving noises of~$X^x_{t}$ and~$X^y_t$ are independent). By Markov property and Theorem~\ref{t:SolEU},
\Bes
\begin{split}
& \ \ \|P_{3 T_0}^{*} \delta_x-P_{3 T_0}^{*} \delta_y\|_{TV}=\frac 12\sup_{\|\phi\|_0 \leq 1} |\E[P_{T_0} \phi(X^x_{2T_0})-P_{T_0} \phi(X^y_{2T_0})]| \\
&\leq \left[1-\PP \{ X^{x}_{2T_0} \in B(r),  X^{y}_{2T_0} \in B(r)\}\right] \\
& \ \ +\frac 12\E\left\{\sup_{\|\phi\|_0 \leq 1} |P_{T_0} \phi(X^x_{2T_0})-P_{T_0} \phi(X^y_{2T_0})| X^{x}_{2T_0} \in B(r),  X^{y}_{2T_0} \in B(r)\right\} \\
& \leq 1-\PP \{ X^{x}_{2T_0} \in B(r),X^{y}_{2T_0} \in B(r)\}+Cr \PP\{X^x_{2T_0} \in B(r),X^y_{2T_0} \in B(r)\} \\
& \leq 1-(1-Cr)\delta^2.
\end{split}
\Ees
Taking $r>0$ sufficiently small, we complete the proof.

\section{Proof of  Theorem~\ref{t:MaiThm} by coupling, $\dim H= \infty$}
\label{s4}

 In this section, we shall prove Theorem~\ref{t:MaiThm} by the Doeblin coupling argument, which gives much more intuitions for
 understanding the way that the dynamics converges
 to the ergodic measure.

\subsection{Construction of the coupling chain}
\label{s6.1}
 Let us first give some preliminary about maximal coupling.
\begin{defn}
Let $\mu_1$, $\mu_2 \in \mcl P(H)$. A pair of random variables $(\xi_1,\xi_2)$ defined on the same probability space is called a \emph{coupling} for $(\mu_1,\mu_2)$ if $\mcl D(\xi_i)=\mu_i$ for $i=1,2$, where $\mcl D(\cdot)$ denotes the distribution of random variable. A coupling $(\xi_1, \xi_2)$ is said to be \emph{maximal} if
\begin{\eqn} \label{e:InqTV}
\PP\{\xi_1 \neq \xi_2\}=\|\mu_1-\mu_2\|_{\rm TV},
\end{\eqn}
and the random variable~$\xi_1$ and~$\xi_2$ conditioned on the event $N:=\{\xi_1 \neq \xi_2\}$ are independent. The latter
condition means that, for any $A_1, A_2 \in \mcl B(H)$, one has
$$
\PP\bigl(\{\xi_1\in A_1\} \cap \{\xi_2\in A_2\}\,|\, N\bigr)
=\PP\bigl(\xi_1\in A_1\,|\,N\bigr)\,\PP\bigl(\xi_2\in A_2\,|\,N\bigr).
$$
\end{defn}

In what follows, we shall the need the following lemma whose proof can be found in~\cite{thorisson2000, lindvall2002, KS10}.

\begin{lem} \label{l:ExiMaxCou}
For any two measures $\mu_1, \mu_2 \in \mcl P(H)$, there exists a
maximal coupling. Moreover, if~$(\xi_1,\xi_2)$ is a maximal
coupling,  then we have\, \footnote{Inequality~\eqref{e:IndN} is
true for any pair of random variables that are independent
conditioned on the event~$\{\xi_1\ne\xi_2\}$.}
\begin{\eqn} \label{e:IndN}
\PP(\xi_1 \in A, \xi_2 \in A)\geq \PP(\xi_1 \in A) \,\PP(\xi_2 \in
A), \ \ \ \ \forall \  A \in \mcl B(H).
\end{\eqn}
\end{lem}

Now let us construct an auxiliary Markov chain in the extended phase space~$H\times H$. Let $T>0$ be some fixed real number to be chosen later.
For any $x:=(x_1, x_2) \in H \times H$, denote
by $M(x)=(M_1(x),M_2(x))$ the maximal coupling of
$(P_T)^* \delta_{x_1}$ and $(P_T)^* \delta_{x_2}$. Let us define a transition function
$\tilde P_T(x,\cdot)$ on the space $H\times H$ such that
\Bes
\tilde P_T(x; A_1 \times A_2)=\begin{cases}
P_T(x_1, A_1 \cap A_2) \ \ {\rm if}  \ \ x_1=x_2,\\
\mcl D(M_1(x),M_2(x))(A_1 \times A_2) \ \ {\rm if} \ x_1,x_2 \in B(r) \ {\rm with} \ x_1 \neq x_2, \\
P_T(x_1, A_1)P_T(x_2,A_2) \ \ \ {\rm otherwise},
\end{cases}
\Ees
where $A_1, A_2 \in \mcl B(H)$ are arbitrary sets, $P_T(x_i,\cdot)$ is the transition probability of $X^{x_i}_T$ for $i=1,2$, and~$\mcl D(\cdot)$ denotes
the distribution of a random variable.
 For any $A \in \mcl B(H \times H)$, $\tilde P_T(x,A)$ is uniquely defined by a classical approximation procedure. Now the transition function~$\tilde P_T(x,\cdot)$ is well defined.

\subsection{Hitting times ${\tau}^{\e} $ and $\tau$}

We denote by $(X_1(kT),X_2(kT))_{k \in \Z^{+}}$ the Markov chain
whose transition function is equal to~$\tilde P_T(x,\cdot)$; here
$\Z^{+} = \{0,1,2, \ldots \}$. Clearly, for each $i=1,2$, $(X_i(kT))$
is also a Markov chain and has the same distribution as
$(X^{x_i}_{kT})$. We shall write $X(kT)=(X_1(kT),X_2(kT))$ for $k
\in \Z^+$.

% {\bf  The remark was deleted!}

\vskip 1mm  For any $r, M>0$, define the hitting  times
\begin{\eqn} \label{e:TauM}
{\tau}^{\e} =\inf\{kT; |X_1(kT)|_\e+|X_2(kT)|_\e \leq M\},
\end{\eqn}
\begin{\eqn} \label{e:TauR}
\tau=\inf\{kT; |X_1(kT)|+|X_2(kT)| \leq r\},
\end{\eqn}
where $\e \in (0,1)$ is the constant in Assumption~\ref{a:A}. Recall that the infimum  over an empty set is equal to~$+\infty$.

\subsubsection{Estimates of the hitting time ${\tau}^{\e} $}

The main result of this subsection is the following theorem, which is in fact a step for estimating~$\tau$.

\begin{thm} \label{t:TauMEst}
For any $p \in (0, \alpha)$ and sufficiently large~$T>0$ there is a constant  $M=M(p,T,\alpha,\beta,\gamma,\e)$ such that,
 for any $x= (x_1, x_2) \in H\times H$,
\begin{\eqn} \label{e:TauMEst}
\E_x \,[e^{\eta {\tau}^{\e} }]\leq C\bg(1+|x_1|^p+|x_2|^p\bg)
\end{\eqn}
where $\eta>0$ is sufficiently small, and $C=C(p, T,\alpha, \beta, \gamma, \e, \|F\|_{0},\eta)$
\end{thm}

To prove Theorem~\ref{t:TauMEst}, we first establish two auxiliary lemmas.

\begin{lem} \label{l:4.5}
For any $p \in (0,\alpha)$, the Markov chain  $(X(kT))$ satisfies the inequality
\begin{\eqn*} \label{e:XkTHe}
\E_x \bg(|X_1(T)|^p_\e+|X_2(T)|^p_\e \bg) \leq C_1 e^{-p
\gamma_1 T} \bg(|x_1|^p_\e+|x_2|^p_\e\bg)+2C_2,
\end{\eqn*}
where $C_1$ and $C_2$ are the same as in Lemma~\ref{l:XtHe}.
\end{lem}

\begin{proof}
By definition of coupling and Lemma~\ref{l:XtHe}, we have

\begin{\eqn*} %\label{e:XiKT}
\E_x |X_i(T)|^p_\e=\E |X^{x_i}_T|^p_\e \leq C_1(p)e^{-p \gamma_1 T}
|x_i|^p_\e+C_2
\end{\eqn*}
\\
for $i=1,2$. From
the above inequality, we complete the proof.
\end{proof}

\vskip 1mm

\begin{lem} \label{l:TauPro}
For any $p \in (0, \alpha)$ and sufficiently large~$T>0$, there
exist positive constants $q=q(p,\gamma) \in (0,1)$  and $M=M(p,T,\alpha,
\beta, \gamma, \|F\|_0,\e)$ such that
\begin{\eqn} \label{e:TauPro}
\PP_x({\tau}^{\e} >kT) \leq q^k\left(1+|x_1|^p_\e+|x_2|^p_\e\right) \quad\mbox{for any $x=(x_1,x_2) \in H^\e \times H^\e$}.
\end{\eqn}
\end{lem}
\vskip 1mm

\begin{proof}
The proof follows the idea in~\cite{De10}. Let us take $T>0$ so large that the coefficient in front of~$|x|_\e^p$ in inequality~\eqref{e:XtHe} is smaller than~$1$. In this case, setting $\PP=\PP_x$, $\E=\E_x$, and
$$
|x|^p_\e=|x_1|^p_\e+|x_2|^p_\e,$$
we can write
\begin{\eqn} \label{e:Sto1}
\begin{split}
 \E\big[|X(kT+T)|_{\e}^p\ \big |\mcl
F_{kT}\big] \leq q^2\, |X(kT)|_{\e}^p+2C_2
\end{split}
\end{\eqn}
where $q>0$ is defined by the relation $q^2=C_1e^{-p\gamma_1T}<1$. By Chebyshev inequality,

\begin{\eqn} \label{e:Sto2}
\PP\left(|X(kT+T)|_{\e}>M|\mcl F_{kT}\right)
\leq \frac{q^2}{M^p}|X(kT)|^p_{\e}+\frac {2 C_2}{M^p}.
\end{\eqn}
\\
Denote
$$B_k=\{|X(jT)|_{\e}>M; j=0, \ldots, k\}$$
and
$$p_k=\PP(B_{k}), \ \ e_{k}=\E \big(|X(kT)|_{\e}^p
\, 1_{B_{k}}\big),$$ integrating~\eqref{e:Sto2} over $B_k$, one has
\begin{\eqn} \label{e:Sto3}
p_{k+1} \leq \frac{q^2}{M^p} e_k+\frac {2C_2}{M^p} p_k.
\end{\eqn}
Moreover, by integrating~\eqref{e:Sto1} over $B_{k}$,

\begin{\eqn} \label{e:Sto4}
%\begin{split}
e_{k+1} \leq \E \big(|X(kT+T)|_{\e}^p 1_{B_{k}}\big)
\leq q^2e_k+2 C_2 p_k.
\end{\eqn}
\\
From~\eqref{e:Sto3} and~\eqref{e:Sto4}, one has
\begin{\eqn} \label{e:MatEP}
\left(
\begin{array}{c}
 e_{k+1} \\
 p_{k+1}
\end{array}
\right) \leq \left( \begin{array}{cc}
q^2 & 2C_2 \\
\frac{q^2}{M^p} & \frac{2C_2}{M^p} \end{array} \right) \left(
\begin{array}{c}
e_{k} \\
p_{k}
\end{array}
\right),
\end{\eqn}
\\
which clearly implies
\Be \label{e:RelEP}
q^2 e_{k+1}+2 C_2 p_{k+1} \leq \left(q^2+\frac{2C_2}{M^p}\right) (q^2 e_k+2C_2 p_k)
\Ee
We can choose $M=M(p, T, \alpha, \beta, \gamma, \e, \|F\|_{0})$ so
that
$$
q^2+2 C_2/M^p \leq q.
$$
Thus we clearly have from~\eqref{e:RelEP}
\begin{\eqn*}
q^2 e_{k}+2C_2 p_{k} \leq q^{k} \left(q^2 e_0+2C_2 p_0\right),
\end{\eqn*}
This inequality, together with the easy fact $p_k=\PP_x({\tau}^{\e} >kT)$,  immediately implies the required estimate~\eqref{e:TauPro} since
$C_2>1$ in inequality~\eqref{e:XtHe}.
\end{proof}

\begin{proof} [Proof of Theorem~\ref{t:TauMEst}]
By the definition of coupling and~\eqref{e:XtHe1}, for any $p \in
(0, \alpha)$ we have

\begin{\eqn} \E_x \bg(|X_1(T)|^p_{\e}+|X_2(T)|^p_\e\bg)=\E|X^{x_1}_T|^p_{\e}+\E |X^{x_2}_T|^p_\e \leq C_4
\left(1+|x_1|^p+|x_2|^p_\e\right)
\end{\eqn}
\\
where $C_4=C_4(p, T, \alpha, \beta, \gamma, \e, \|F\|_0)$.
\vskip1mm

  For any $x=(x_1,x_2) \in H \times H$, by Markov property,~\eqref{e:TauPro} and the above inequality,
 we easily have

 \begin{align}
 \begin{split}
 \E_x \left[e^{\eta  \tau^\e}\right]&=\E_x \left(e^{\eta  \tau^\e} 1_{\{\tau^\e \leq
 T\}}\right)+\E_x \left(e^{\eta \tau^\e} 1_{\{\tau^\e>
 T\}}\right) \\
 & \leq e^{\eta T}+\E_x \left\{1_{\{\tau^\e>
 T\}} \E_{X(T)} \left[e^{\eta \tau^\e}\right]\right\} \\
 &\leq e^{\eta T}+  C_5  \E_x \left[1+|X_1(T)|^{p}_{\e}+|X_2(T)|_\e^p\right]  \\
 & \leq C_6(1+|x_1|^{p}+|x_2|^p)
\end{split}
\end{align}
\\
where $C_i=C_i(p, \alpha,\eta,\gamma, \beta, \e,\|F\|_{0},T) \ (i=5,6)$.

\end{proof}

\subsubsection{Estimates of the hitting time  $\tau$}
\begin{thm} \label{t:TauREst}
For any $p \in (0, \alpha)$ and sufficiently large~$T>0$, there exist positive constants
$\lambda=\lambda(T, p, \alpha,\beta,\gamma, \|F\|_{0},r)$ and
$C=C(p, \alpha,\beta,\gamma, \|F\|_{0}, r, T)$ such that
\begin{\eqn} \label{e:TauREst}
\E_x[e^{\lambda \tau}]\leq C(1+|x_1|^{p}+|x_2|^p).
\end{\eqn}
\end{thm}
\ \ \

 The key point of the proof is to use Theorem~\ref{t:TauMEst} and
Lemma~\ref{l:Irr2} below. The argument is quite general, for
simplicity, let us give its heuristic idea by using $(X_{kT})$, (note the
difference between $X_{kT}$ and $X(kT)$), as follows:
\begin{itemize}
\item [(i)] Since $B_\e(M)$ is compact in $H$, by
irreducibility and uniform strong Feller property we have that $\inf_{z \in
B_\e(0,M)} P_T(z, B(r))=p>0$. Therefore, as long as $X_{kT}$ is in
$B_\e(M)$, it has the probability at least $p$ to jump into $B(r)$
at $(k+1)T$.
\item [(ii)] Suppose that $(X_{kT})$ enters $B_\e(M)$ for $j$ times
\emph{before} it jumps into $B(r)$, by strong Markov property
and (i) this event happens with some probability less than
$(1-p)^j$.
\item [(iii)] If
$\tau=kT$ for some large $kT$ (i.e. the process first enters $B(r)$
at $kT$), $j$ is also large. Thus $\PP(\tau=kT) \leq (1-p)^j$ is
small.
\end{itemize}
\ \ \

Let us now make the above heuristic argument rigorous for $(X(kT))$.
We first need to establish the following lemma.
\begin{lem} \label{l:Irr2}
For any compact set ${\mathcal K}\subset H\times H$ and any $R>0$, there exists
some constant $\delta=\delta({\mathcal K},R)>0$ such that

\begin{\eqn}  \label{e:Irr2}
\inf_{x\in{\mathcal K}} \PP_x \{X(T) \in B(R) \times B(R)\}>0.
\end{\eqn}
\end{lem}

\begin{proof}
To show~\eqref{e:Irr2}, we split the
 argument
 into the following three cases.

 (i) As $x \notin
B(r) \times B(r)$ with $x_1 \neq x_2$,
$X_1(T)$ and $X_2(T)$ are independent. Therefore, by Lemma
\ref{p:Irr} one has

\begin{\eqn*}
\begin{split} \PP_x (X(T) \in B(R) \times B(R))& =\PP_{x}\left(X_1 (T) \in
B(R)\right) \PP_{x}\left(X_2 (T) \in B(R)\right)
\\
&=\PP\left(X^{x_1}_T \in B(R)\right) \PP\left(X^{x_2}_T \in
B(R)\right)>0.
\end{split}
\end{\eqn*}
\\
(ii) As $x=(x_1,x_2)$ with $x_1=x_2$, we have $X_1(T)=X_2(T)$.
Hence,

\begin{\eqn*}
 \PP_x (X(T) \in B(R) \times B(R))=\PP\left(X^{x_1}_T \in B(R)\right)>0.
\end{\eqn*}
\\
(iii) As $x \in B(r) \times B(r)$ with $x_1 \neq x_2$, by the
maximal coupling property~\eqref{e:IndN} one has

\begin{\eqn*}
\begin{split}
\PP_x(X(T) \in B(R) \times B(R)) &=\PP_x(M(x) \in B(R) \times B(R)) \\
& \geq \PP_x(M_1(x) \in B(R))
\PP_x(M_2(x) \in B(R)) \\
&=\PP(X^{x_1}_T \in B(R)) \PP(X^{x_2}_T \in B(R)) >0,
\end{split}
\end{\eqn*}
where $M(x)=(M_1(x),M_2(x))$ is the maximal coupling of $(P^{*}_T
\delta_{x_1}, P^{*}_T \delta_{x_2})$.

>From (i)-(iii) it is clear that

\begin{\eqn*}
\PP_x (X(T) \in B(R) \times B(R)) \geq \PP(X^{x_1}_T \in B(R))
\PP(X^{x_2}_T \in B(R)).
\end{\eqn*}
\\
By Feller property of $P_T$ and Lemma~\ref{p:Irr}, for any open subset $O\subset H$ the function $x\mapsto P_T(x, O)$ is positive and lower
semi-continuous. Hence, it is separated from zero on any compact subset. Therefore, there is a constant
$\delta=\delta(x,R,T)>0$ so that
\begin{\eqn}
\inf_{x\in\mathcal K}\PP(X^{x_1}_T \in B(R)) \PP(X^{x_2}_T \in
B(R))>0.
\end{\eqn}
>From the above two inequality, we complete the proof.
\end{proof}

\begin{proof} [Proof of Theorem~\ref{t:TauREst}]
Take $M=M(p, T, \alpha, \beta, \gamma, \e, \|F\|_{0})$ defined in
Theorem~\ref{t:TauMEst}, and \emph{simply write}
$$|x|^p=|x_1|^p+|x_2|^p, \ \ \ \ \ x=(x_1, x_2) \in H \times H.$$
Let us prove the
theorem in the following four steps: \\

 \emph{Step
1}. Write ${\tau}^{\e} _0=0$, $\tau^\e_1=\tau^\e$ and define
\begin{\eqn*}
{\tau}^{\e} _{k+1}=\inf\{jT>{\tau}^{\e} _k; |X_1(jT)|_\e+|X_2(jT)|_\e \leq M \}
\end{\eqn*}
for all integer $k \geq 1$. Since $(X(kT))$ is a discrete time
Markov chain, it is strong Markovian. By Theorem~\ref{t:TauMEst} and
Poincare inequality $|z| \leq \frac{1}{\gamma^\e_1}|z|_\e$ for any
$z \in H^\e$, we have

\begin{\eqn} \label{e:TauKK-1}
\E_{X({\tau}^{\e} _k)}\left[e^{\eta ({\tau}^{\e} _{k+1}-{\tau}^{\e} _k)}\right] \leq C (1+|X({\tau}^{\e} _k)|^{p}) \leq
c(1+M^{p}),
\end{\eqn}
\\
where $c=C \left(1+2^{p}/\gamma^{\e p}_1\right)$ and $C=C(p,
\alpha,\beta,\gamma, \|F\|_{0}, r, T)$ is the same as in Theorem
\ref{t:TauMEst}. The above inequality, together with strong Markov
property, implies

\begin{\eqn} \label{e:TauKEst}
\begin{split}
\E_x[e^{\eta {\tau}^{\e} _k}] &=\E_x \left[e^{\eta {\tau}^{\e} _1}\E_{X({\tau}^{\e} _1)}\left[e^{\eta ({\tau}^{\e} _2-\tau^\e_1)}\cdots \E_{X(
\tau^\e_{k-1})}\left[e^{\eta ({\tau}^{\e} _k-{\tau}^{\e} _{k-1})}\right]\cdots \right]\right]\\
& \leq c^k (1+M^{p})^{k-1} (1+|x|^{p}).
\end{split}
\end{\eqn}
\vskip 3mm

 \emph{Step 2}.
 Since $B_\e(M) \subset \subset H$, by Lemma~\ref{l:Irr2} we have

\begin{\eqn*} %\label{e:UniIrr}
\inf_{y \in B_\e(M) \times B_\e(M)} \PP_y \big(X(T) \in
B(r) \times B(r) \big)=\sigma,
\end{\eqn*}
\\
for all $r>0$, where $\sigma=\sigma(\e,M,r,T)>0$. Therefore, for some $\sigma \in (0,1)$,

\begin{\eqn} \label{e:UniIrr}
\inf_{|y|_\e \leq M} \PP_y \big(X(T) \in B(r) \times B(r) \big) \geq
\sigma,
\end{\eqn}
\\
where $|y|_\e=|y_1|_\e+|y_2|_\e$.
 \vskip 2mm

 \emph{Step 3}. Given any $k \in \N$, define
$$\rho_k=\sup\{j; \ \tau^\e_j \leq kT\}.$$
Clearly, ${\tau}^{\e} _{\rho_k+1}>kT$. For any $k \in \N$, one has
\begin{\eqn}
\begin{split}
\PP_x(\tau=kT)&=\sum_{j=0}^k \PP_x(\tau=kT, \rho_k=j) \\
&=\sum_{j=0}^l \PP_x(\tau=kT, \rho_k=j)+\sum_{j=l+1}^k
\PP_x(\tau=kT, \rho_k=j) \\
&=:I_1+I_2
\end{split}
\end{\eqn}
where $l<k$ is some integer number to be chosen later.
\vskip 1mm

\emph{Step 4}. Let us estimate the above $I_1$ and $I_2$. By the
definition of $\rho_k$, Chebyshev inequality and strong Markov
property, we have

\begin{\eqn*}
\begin{split}
\PP_x(\tau=kT, \rho_k=j) &\leq \PP_x \left(
\tau^\e_j>kT/2\right)+\PP_x\left(\tau^\e_j \leq kT/2, \ \rho_k=j \right) \\
&\leq \PP_x \left(\tau^\e_j>kT/2\right)+\PP_x\left(\tau^\e_j \leq kT/2, \ {\tau}^{\e} _{j+1}>kT \right)
\\
& \leq e^{-\eta kT/2} \E_x\left[e^{\eta {\tau}^{\e} _j}\right]+\E_x \left[\PP_{X({\tau}^{\e} _j)} \left({\tau}^{\e} _{j+1}-
\tau^\e_j>kT/2\right) \right]
\end{split}
\end{\eqn*}
\\
By~\eqref{e:TauKEst} and~\eqref{e:TauKK-1}, the above inequality
implies
\begin{\eqn*}
\PP_x(\tau=kT, \rho_k=j) \leq c^j(1+M^{p})^{j-1} (1+|x|^{p})e^{-\eta
kT/2}+c(1+M^{p}) e^{-\eta kT/2}.
\end{\eqn*}
 Hence,
\begin{\eqn}
\begin{split}
 I_1 & \leq  \left[c^{l+1}(1+M^{p})^{l+1} (1+|x|^{p})+l c(1+M^{p})
\right]e^{-\eta kT/2} \\
& \leq  c^{l+2} (1+M^{p})^{l+2} (1+|x|^{p}) e^{-\eta kT/2}.
\end{split}
\end{\eqn} \ \ \

Now we estimate $I_2$. For $j>l$, by the definitions of $\tau$ and
$\rho_k$, strong Markov property and~\eqref{e:UniIrr}, we have
\begin{\eqn*}
\begin{split} \PP_x\left(\tau=kT, \rho_k=j\right) \leq
\PP_x\left(|X({\tau}^{\e} _1)|>r, \ldots, |X({\tau}^{\e} _j)|>r\right) \leq (1-\sigma)^j.
\end{split}
\end{\eqn*}
 Hence,
\begin{\eqn}
I_2 \leq \frac{1}{\sigma}(1-\sigma)^{l+1}.
\end{\eqn}
Taking $\bar \eta=\frac{\eta}{4\log (c+cM^{p})}$ and $l=[\bar \eta k
T]$, we have
\begin{\eqn*}
\begin{split}
 I_1 \leq e^{-k \eta T/4}\bg(1+|x|^{p}\bg), \ \ I_2 \leq \frac
1\sigma \exp \big \{-k T \bar \eta \log \frac 1{1-\sigma}\big\}.
\end{split}
\end{\eqn*}
Combining the above estimates of $I_1$ and $I_2$, and taking
$2\lambda=\frac{\eta} 4 \wedge \bar \eta \log \frac 1{1-\sigma}$, we
have
\begin{equation*}
\PP_x(\tau=kT) \leq \left(c^2+\frac 1\sigma\right) e^{-2 \lambda kT}
\left (1+|x|^{p}\right)
\end{equation*}
 From the above inequality, we immediately obtain the desired estimate.
\end{proof}

\subsection{Final part of the coupling proof } \label{s5}
 It is divided into two steps.
\medskip

{\it Step~1}.
By the same reason as in Steps 1 and 2 in Section~\ref{s:Har}, to prove the uniqueness of an invariant measure and inequality~\eqref{2.5}, it suffices to show that
\begin{equation} \label{7.1}
\|P_{kT}(x_1,\cdot)-P_{kT}(x_2,\cdot)\|_{\rm TV}
\le C\,(1+|x_1|^p+|x_2|^p)e^{-ckT},
\quad x_1,x_2\in H,
\end{equation}
where $C$ and~$c$ are positive constants not depending on $x_1$, $x_2$, and~$k$.  Let $(X_1(t),X_2(t))$, $t\in T\Z$, be the chain constructed in Section~\ref{s6.1}. Define the stopping time
$$
\rho=\min\{kT: k\in\N, X_1(kT)=X_2(kT)\},
$$
where the minimum over an empty set is equal to~$+\infty$. Suppose we have proved that
\begin{equation}
\PP_x\{\rho>kT\}\le Ce^{-\eta kT}(1+|x_1|^p+|x_2|^p),\label{5.5}
\end{equation}
where $x=(x_1,x_2)\in H\times H$ is arbitrary, and the positive constants~$\eta$ and~$C$ do not depend on~$x$.
In this case, using the fact that $X_1(kT)=X_2(kT)$ for $k\ge l$ as soon as $X_1(lT)=X_2(lT)$, we can write
\begin{align*}
\bigl|P_{kT}(x_1,\Gamma)-P_{kT}(x_2,\Gamma)\bigr|
&=\bigl|\E_{x}1_\Gamma\bigl(X_1(kT)\bigr)
-\E_{x}1_\Gamma\bigl(X_2(kT)\bigr)\bigr|\\
&=\E_{x}\Bigl(1_{\{\rho>kT\}}\bigl|1_\Gamma\bigl(X_1(kT)\bigr)
-1_\Gamma\bigl(X_2(kT)\bigr)\bigr|\Bigr)\\
&\le \PP_x\{\rho>kT\}.
\end{align*}
Using~\eqref{5.5}, we obtain
$$
\bigl|P_{kT}(x_1,\Gamma)-P_{kT}(x_2,\Gamma)\bigr| \le Ce^{-\eta
kT}\,(1+|x_1|^p+|x_2|^p).
$$
Taking the supremum over all $\Gamma\in {\mathcal B}(H)$, we arrive
at the required inequality~\eqref{5.2}.

\medskip
{\it Step~2}.
Thus, it remains to establish~\eqref{5.5}. To this end, we first note that if~$r>0$ is sufficiently small, then
\begin{\eqn} \label{5.6}
\PP_x \left\{X_1(T)\ne X_2(T)\right\} \leq 1/2\quad
\mbox{for any $x\in B(r) \times B(r)$}.
\end{\eqn}
Indeed, by Theorem~\ref{t:Sol}, for any function $f\in B_b(H)$
with $\|f\|_0\le1$ we have
\begin{\eqn*}
\bigl|(P_T(x_1,\cdot),f)-(P_T(x_2,\cdot),f)\bigr|
=|P_Tf(x_1)-P_Tf(x_2)| \leq C_1 |x_1-x_2|\quad\mbox{for $x_1,x_2 \in H$}.
\end{\eqn*}

Recalling the definition of the total variation distance, we see that
$$
\|P_T(x_1,\cdot)-P_T(x_2,\cdot)\|_{\mathrm TV}\le1/2,
\quad x_1,x_2\in B(r),
$$
where $r>0$ is sufficiently small. Since $\bigl(X_1(T),X_2(T)\bigr)$
is a maximal coupling for the pair
$\bigl(P_T(x_1,\cdot),P_T(x_2,\cdot)\bigr)$, by~\eqref{e:InqTV} we
arrive at~\eqref{5.6}.

\smallskip
We now introduce the iterations~$\{\tau_n\}$ of the stopping time~$\tau$ defined by~\eqref{e:TauR}:
$$
\tau_1=\tau, \quad
\tau_{n+1}=\inf \left \{j T >\tau_{n}: |X_1(jT)|+|X_2(jT)| \leq r \right \}.
$$
An argument similar to that used in Step~1 of the proof of Theorem~\ref{t:TauREst} shows that
$$
\E_xe^{\lambda \tau_n}\le K^n(1+|x_1|^p+|x_2|^p),
$$
where $K>1$ and $\lambda>0$ do not depend on~$x_1,x_2\in H$ and $n\ge1$. By the Chebyshev inequality, it follows that
\begin{equation} \label{5.7}
\PP_x\{\tau_n>kT\}\le e^{-\lambda kT}K^n(1+|x_1|^p+|x_2|^p).
\end{equation}
Let us define the events
$$
\Gamma_n=\{X_1(\tau_m+T)\ne X_2(\tau_m+T)\mbox{ for $1\le m\le n$}\}
$$
and set $P_n(x)=\PP_x(\Gamma_n)$. By~\eqref{5.6} and the strong Markov property, we have
$$
\PP_x\bigl\{X_1(\tau_n+T)\ne X_2(\tau_n+T)\,|\,{\mathcal F}_{\tau_n}\bigr\}
\le \PP_{X(\tau_n)}\{X_1(T)\ne X_2(T)\} \leq 1/2
$$
It follows that
\begin{align*}
P_n(x)
&=\PP_x\bigl(\Gamma_{n-1}\cap\{X_1(\tau_n+T)\ne X_2(\tau_n+T)\}\bigr)\\
&=\E_x\bigl(1_{\Gamma_{n-1}}
\PP_x\{X_1(\tau_n+T)\ne X_2(\tau_n+T)\,|\,{\mathcal F}_{\tau_n}\}\bigr)
\le \frac12 P_{n-1}(x),
\end{align*}
whence, by iteration, we get $P_n(x)\le2^{-n}$ for any $n\ge1$. Combining this with~\eqref{5.7}, for any integers $n,k\ge1$ we obtain
\begin{align*}
\PP_x\{\rho>kT\}
&=\PP_x\{\rho>kT,\tau_n< kT\}+\PP_x\{\rho>kT,\tau_n\ge kT\}\\
&\le \PP_x(\Gamma_n)+\PP_x\{\tau_n\ge kT\}\\
&\le 2^{-n}+e^{-\lambda kT}K^n(1+|x_1|^p+|x_2|^p).
\end{align*}
Taking $n=\e k$ with a sufficiently small~$\e>0$, we arrive at the required inequality~\eqref{5.5}. The proof of Theorem~\ref{t:MaiThm} is complete.

\section{Proofs of exponential mixing when $\dim H< \infty$} \label{s6}
 First of all, by Theorem 2.5 of~\cite{PXZ10},
the system in~\eqref{strong} has at least one invariant measure. To prove Theorem~\ref{t:ErgFin}, we can use the Harris method or the  coupling argument.

In both approaches we need also the decay estimates for solutions
 given in Lemmas~\ref{l:XtHe} and~\ref{vai}. These
can be easily adapted to the strong solution~$X_t$
in~\eqref{strong} (indeed, by the Gronwall lemma, starting from~\eqref{strong}, we get $\E |Z_A(t)|^p <\infty$ for any $p \in (0, \alpha)$).

For the Harris approach, in order to verify  the two conditions in
Theorem~\ref{t:Har} we can
repeat the same argument as in Section~\ref{s:Har}.

\vskip 1mm For the coupling approach, the key point
 is irreducibility and gradient estimates of
%estimate~\eqref{grad12} and
 Theorem~\ref{t:Sol}. Using a
similar (but easier) argument as in Section~\ref{s4}, we can prove
Theorem~\ref{t:ErgFin} in the following three steps:
\begin{itemize}
\item[(1)]
constructing the coupling and defining the stopping
time $\tau$ exactly as in Section~\ref{s6.1};
\item[(2)]
proving the exponential estimate~\eqref{e:TauREst};
\item[(3)]
using the same
argument as in Section~\ref{s5} which involves  the coupling time.
\end{itemize}
Finally, let us emphasize that unlike the infinite-dimensional setting, we do not need to introduce $H^\e$ and ${\tau}^{\e}$ to get some
compactness, since any finite-dimensional closed ball is automatically compact.
%\bibliographystyle{amsplain}
%\begin{thebibliography}{99}
%\bibliography{PSXZ}

%%%%%%%%%%%%%%%%%%%%%%%%%%%%%%%%%%%%%%%%%%

\bibliographystyle{amsplain}

\end{document}